\documentclass[reqno]{amsart}
\usepackage{latexsym,amsmath,amsfonts,amscd,amssymb}

\usepackage{eepic}
\theoremstyle{plain}
\newtheorem{theorem}{Theorem}[section]
\newtheorem{corollary}[theorem]{Corollary}
\newtheorem{lemma}[theorem]{Lemma}
\newtheorem{proposition}[theorem]{Proposition}
\theoremstyle{definition}
\newtheorem{definition}[theorem]{Definition}
\newtheorem{remark}[theorem]{Remark}
\numberwithin{equation}{section}

\newcommand{\flechad}[1]{
\mbox{$\begin{picture}(#1,6) \put(1,3){\vector(1,0){#1}}
\end{picture}\hspace{2pt}$}}

\newcommand{\s}{\sigma}

\newcommand{\cC}{\mathcal{C}}

\newcommand{\cN}{\mathcal{N}}

\newcommand{\cO}{\mathcal{O}}

\newcommand{\cS}{\mathcal{S}}

\newcommand{\RR}{\mathbb{R}}
\newcommand{\CC}{\mathbb{C}}
\newcommand{\HH}{\mathbb{H}}
\newcommand{\PP}{\mathbb{P}}
\newcommand{\ZZ}{\mathbb{Z}}

\newcommand{\inc}{\hookrightarrow}
\newcommand{\lto}{\longrightarrow}
\newcommand{\surj}{\twoheadrightarrow}

\newcommand{\x}{\times}
\newcommand{\ox}{\otimes}

\newcommand{\frM}{{\frak M}}
\newcommand{\coeff}{\mathop{\mathrm{coeff}}}
\newcommand{\Sym}{\mathrm{Sym}}
\newcommand{\Res}{\mathrm{Res}}
\newcommand{\GCD}{\mathrm{gcd}}
\newcommand{\Jac}{\mathrm{Jac}}
\DeclareMathOperator{\rk}{rk} 

\DeclareMathOperator{\Hom}{Hom} \DeclareMathOperator{\Ext}{Ext}
 \DeclareMathOperator{\Gr}{Gr}
\DeclareMathOperator{\GL}{GL}\DeclareMathOperator{\Aut}{Aut}

\newcommand{\scp}{{\s_c^+}}
\newcommand{\smp}{{\s_m^+}}
\newcommand{\scm}{{\s_c^-}}

\newcommand{\moduli}{\cN_\s}

\newcommand{\modulimenos}{\cN_{\s_c^-}}

\newcommand{\Smenos}{\cS_{\scm}}

\newcommand{\hn}{{\bar{n}_0}}

\title{Hodge polynomials of the moduli spaces of rank $3$ pairs}

\subjclass[2000]{Primary: 14F45. Secondary: 14D20, 14H60.}
\keywords{Moduli space, complex curve, stable triple, Hodge
polynomial.}

\author{Vicente Mu\~noz}
  \address{Instituto de Ciencias Matem\'aticas CSIC-UAM-UCM-UC3M \\
  Consejo Superior de Investigaciones Cient\'{\i}ficas \\ Serrano 113 bis
  \\ 28006 Madrid \\ Spain}
    \address{Facultad de Matem\'{a}ticas \\ Universidad Complutense
  de Madrid \\ Plaza de Ciencias 3
  \\ 28040 Madrid \\ Spain}
  \email{vicente.munoz@imaff.cfmac.csic.es}

\thanks{Partially supported through grant MEC
(Spain) MTM2004-07090-C03-01}

\begin{document}
\maketitle

\begin{abstract}
 Let $X$ be a smooth projective curve of genus $g\geq 2$ over the
 complex numbers. A holomorphic triple $(E_1,E_2,\phi)$ on $X$ consists of two
 holomorphic vector bundles $E_{1}$ and $E_{2}$ over $X$ and a
 holomorphic map $\phi \colon E_{2} \to E_{1}$.
 There is a concept of stability for triples which depends on a
 real parameter $\sigma$.
 In this paper, we determine the Hodge polynomials of the moduli
 spaces of $\sigma$-stable triples with $\rk(E_1)=3$, $\rk(E_2)=1$, using the theory of mixed Hodge
 structures.
 This gives in particular the Poincar\'{e} polynomials
 of these moduli spaces. 
 As a byproduct, we recover the Hodge
 polynomial of the moduli space of odd degree rank $3$ stable
 vector bundles.
\end{abstract}

\section{Introduction}
\label{sec:introduction}

Let $X$ be a smooth projective curve of genus $g\geq 2$ over the
field of complex numbers. A holomorphic triple $T =
(E_{1},E_{2},\phi)$ on $X$ 
consists of two holomorphic vector bundles $E_{1}$ and $E_{2}$
over $X$, of ranks $n_1$ and $n_2$ and degrees $d_1$ and $d_2$
respectively, and a holomorphic map $\phi \colon E_{2} \to E_{1}$.
We call $(n_1,n_2,d_1,d_2)$ the type of the triple. There is a
concept of stability for a triple which depends on the choice of a
parameter $\s \in \RR$. This gives a collection of moduli spaces
$\moduli$, which have been studied in
\cite{BGP,BGPG,GPGM,MOV1,MOV2}. The range of the parameter $\s$ is
an interval $I\subset \RR$ split by a finite number of
\textit{critical values} $\s_c$ in such a way that, when $\s$
moves without crossing a critical value, then $\moduli$ remains
unchanged, but when $\s$ crosses a critical value, $\moduli$
undergoes a transformation which we call \textit{flip}. The study
of this process allows to obtain geometrical information on all
the moduli spaces $\moduli$.

For a projective smooth variety $Z$, the Hodge polynomial is
defined as
 $$
 e(Z)(u,v)=\sum_{p,q} h^{p,q}(Z) u^pv^q\, ,
 $$
where $h^{p,q}(Z)=\dim H^{p,q}(Z)$ are the Hodge numbers of $Z$.
In particular, the Poincar\'{e} polynomial equals $P_t(Z)=e(Z)(t,t)$.
The theory of mixed Hodge structures introduced by Deligne
\cite{De} allows to extend the definition of Hodge polynomials to
any algebraic variety (non-smooth or non-complete). Using this, in
\cite{MOV1,MOV2} the Hodge polynomials of the moduli spaces of
triples when the ranks of $E_1$ and $E_2$ are at most $2$ were
found.

When the rank of $E_2$ is one, we have the so-called
\textit{pairs}, studied in \cite{Th,GP,MOV1}. The moduli spaces of
pairs are smooth projective varieties for non-critical values of
$\s$. In \cite{MOV1} it was computed the Hodge polynomials of the
moduli spaces of rank $2$ pairs, and this was applied to recover
the Hodge polynomial of the moduli space of rank $2$ odd degree
stable bundles.
Our expectation is that the current technique can be used to:

\begin{itemize}
 \item[(i)] find the Hodge polynomials of the moduli spaces of rank
$n$ pairs (and in particular, the Poincar\'{e} polynomials of such
moduli spaces).
 \item[(ii)] recover the Hodge polynomials of the moduli spaces of
 rank $n$ and degree $d$ stable bundles, for $n,d$ coprime. Recursive formulas
 were obtained by Earl and Kirwan \cite{EK}, but our method
 will probably yield more explicit recursive formulas, in
 the spirit of those in \cite{Zagier}. 
\end{itemize}

Here we do the following step, namely $n=3$. New features, which
will appear in the general case, are present now. However there
are further complications when $n>3$ which deserve more thorough
investigation. This has convinced us to work out the case $n=3$
and leave the general case for future research. Our main results
are:

\begin{theorem} \label{thm:main-1}
  Let $\cN_\s=\cN_\s(3,1,d_1,d_2)$ be the moduli space
  of $\s$-stable triples of type $(3,1,d_1,d_2)$.
  Assume that $\s\in I$ is non-critical.
  Then the Hodge polynomial is
  $$
  \begin{aligned}
   e(\cN_\s) =&
   (1+u)^{2g}(1+v)^{2g}\coeff_{x^0}\frac{(1+ux)^{g}(1+vx)^{g}}{(1-x)(1-uvx)x^{d_1-d_2}}
   \cdot
   \\
    & \cdot \Bigg[
   \left(\frac{(uv)^{2d_1-2d_2-2 n_0} x^{n_0}}{1-(uv)^{-2}x} -
    \frac{(uv)^{2g-2-2d_1+3n_0} x^{n_0}}{1-(uv)^{3}x}
    \right)
   \cdot  \frac{(1+u^2v)^{g}(1+uv^2)^g  -(uv)^{g}(1+u)^{g}(1+v)^{g}}
    {(1-uv)^2(1-(uv)^2)} \\ &  +
   \frac{(uv)^{g-1}(1+u)^{g}(1+v)^{g}}{(1-uv)^2(1+uv)}
    \Bigg(
  \frac{(uv)^{2d_1-2d_2-2\hn+1} x^\hn}{(1-(uv)^{-2}x)(1-(uv)^{-1}x)}  \\
    &
    \qquad + \frac{(uv)^{2g-2-2d_1+3\hn} x^{\hn} }{(1-(uv)^{3}x)(1-(uv)^2x)}
    - \frac{(1+uv)(uv)^{g-1-d_2+\hn/2} x^{\hn}}{(1-(uv)^2x)(1-(uv)^{-1}x)}
    \Bigg) \Bigg] \, ,
  \end{aligned}
   $$
  where $n_0=\left[\frac{\s+d_1+d_2}{2}\right]$ and
  $\hn=2[\frac{n_0+1}{2}]$ (where $[x]$ denotes the integer part of $x\in\RR$).

  In particular, the Poincar\'{e} polynomial is
  $$
  \begin{aligned}
    P_t(\cN_\s) =&(1+t)^{4g}
   \coeff_{x^0}\frac{(1+tx)^{2g}}{(1-x)(1-t^2x)x^{d_1-d_2}} \cdot
   \\
    & \cdot \Bigg[
   \left(\frac{t^{4d_1-4d_2-4 n_0} x^{n_0}}{1-t^{-4}x} -
    \frac{t^{4g-4-4d_1+6n_0} x^{n_0}}{1-t^{6}x} \right)
   \cdot  \frac{(1+t^3)^{2g}  -t^{2g}(1+t)^{2g}}
    {(1-t^2)^2(1-t^4)} \\ &  +
   \frac{t^{2g-2}(1+t)^{2g}}{(1-t^2)^2(1+t^2)}
    \Bigg(
  \frac{t^{4d_1-4d_2-4\hn+2} x^\hn}{(1-t^{-4}x)(1-t^{-2}x)}
  + \frac{t^{4g-4-4d_1+6\hn} x^{\hn} }{(1-t^{6}x)(1-t^4x)}
    - \frac{(1+t^2)t^{2g-2-2d_2+\hn} x^{\hn}}{(1-t^4x)(1-t^{-2}x)}
    \Bigg) \Bigg] \, .
  \end{aligned}
  $$
\end{theorem}

Let $M(n,d)$ denote the moduli space of polystable vector bundles
of rank $n$ and degree $d$ over $X$. This moduli space is
projective. We also denote by $M^s(n,d)$ the open subset of stable
bundles, which is smooth of dimension $n^2(g-1)+1$. If
$\GCD(n,d)=1$, then $M(n,d)=M^s(n,d)$.

For the smallest possible values of the parameter $\s\in I$, there
is a map from $\cN_\s(3,1,d_1,d_2)$ to $M(3,d)$ given by
$(E_1,E_2,\phi)\mapsto E_1$. The study of this map allows us to
recover the Hodge polynomial of $M(3,d)$ when $d\not\equiv 0\pmod
3$. This was found previously in \cite{EK} by other methods.

\begin{theorem} \label{thm:main-2}
Assume that $d\not\equiv 0 \pmod 3$. Then the Hodge polynomial of
$M(3,d)$ is
  $$
  \begin{aligned}
   e(M(3,d))
    =& \frac{(1+u)^{g}(1+v)^{g}}{(1-uv)(1-(uv)^2)^2(1-(uv)^3)}
   \Big( -(1+u)^{g}(1+v)^{g}(1+uv)^2(uv)^{2g-1}
   (1+u^2v)^g(1+uv^2)^g \\
   & +(1+u)^{2g}(1+v)^{2g} (uv)^{3g-1}(1+uv+u^2v^2)
   +(1+u^2v^3)^g(1+u^3v^2)^g(1+u^2v)^g(1+uv^2)^g \Big)\, .
  \end{aligned}
   $$

  In particular, the Poincar\'{e} polynomial is
  $$
  P_t(M(3,d))= (1+t)^{2g} \,
    \frac{-(1+t)^{2g}(1+t^2)^2t^{4g-2}
   (1+t^3)^{2g}+(1+t)^{4g}t^{6g-2}(1+t^2+t^4)
   +(1+t^5)^{2g} (1+t^3)^{2g}}{(1-t^2)(1-t^4)^2(1-t^6)}
   \, .
     $$
\end{theorem}

{\bf Acknowledgements:} The author would like to thank Marina
Logares, Peter Gothen, Mar\'{\i}a-Jes\'{u}s V\'{a}zquez-Gallo and Daniel Ortega
for useful conversations. Special thanks to Peter Newstead for 
pointing out a sign mistake in the formula of Theorem \ref{thm:main-2} 
in the published version of the article.

\section{Hodge Polynomials}
\label{sec:virtual}

Let us start by recalling the Hodge-Deligne theory of algebraic
varieties over $\CC$. Let $H$ be a finite-dimensional complex
vector space. A {\em pure Hodge structure of weight $k$} on $H$ is
a decomposition
 $$
 H=\bigoplus\limits_{p+q=k} H^{p,q}
 $$
such that $H^{q,p}=\overline{H}^{p,q}$, the bar denoting complex
conjugation in $H$. We denote
  $$
  h^{p,q}(H)=\dim H^{p,q}\ ,
  $$
which is called the Hodge number of type $(p,q)$. A Hodge
structure of weight $k$ on $H$ gives rise to the so-called {\em
Hodge filtration} $F$ on $H$, where
  $$
  F^p= \bigoplus\limits_{s\geq p} H^{s,p-s}\ ,
  $$
which is a descending filtration. Note that $\Gr_F^p H =
F^p/F^{p+1}= H^{p,q}$.

Let $H$ be a finite-dimensional complex vector space. A {\em
(mixed) Hodge structure} over $H$ consists of an ascending weight
filtration $W$ on $H$ and a descending Hodge filtration $F$ on $H$
such that $F$ induces a pure Hodge filtration of weight $k$ on
each $\Gr^W_k H= W_k/W_{k-1}$. Again we define
  $$
  h^{p,q}(H)=\dim \, H^{p,q}\, , \qquad \text{where}\quad H^{p,q}
  =\Gr_F^p\Gr^W_{p+q} H\, .
  $$

Deligne has shown \cite{De} that, for each complex algebraic
variety $Z$,  the cohomology $H^k(Z)$ and the cohomology with
compact support $H_c^k(Z)$ both carry natural Hodge structures. If
$Z$ is a compact smooth projective variety (hence compact K\"{a}hler)
then the Hodge structure $H^k(Z)$ is pure of weight $k$ and
coincides with the classical Hodge structure given by the Hodge
decomposition of harmonic forms into $(p,q)$ types.

\begin{definition} \label{def:Hodge-poly}
For \emph{any} complex algebraic variety $Z$ (not necessarily
smooth, compact or irreducible), we define the Hodge numbers as
  $$
  h^{k,p,q}_c(Z)=h^{p,q}(H^k_c(Z))=\dim \Gr^p_F \Gr_{p+q}^W H^k_c(Z)\, .
  $$
Introduce the Euler characteristic
 $$
   \chi^{p,q}_c(Z) = \sum (-1)^k h^{k,p,q}_c(Z)
 $$
The \emph{Hodge polynomial} of $Z$ is defined \cite{DK} as
 $$
 e(Z)=e(Z)(u,v)= \sum_{p,q} (-1)^{p+q}\chi_c^{p,q}(Z) u^p v^q\, .
 $$
\end{definition}

If $Z$ is smooth and projective then the mixed Hodge structure on
$H^k(Z)$ is pure of weight $k$, so $\Gr_k^W H_c^k(Z)
=H_c^k(Z)=H^k(Z)$ and the other pieces $\Gr_m^W H_c^k(Z)=0$,
$m\neq k$. So
   $$
   \chi_c^{p,q}(Z)=(-1)^{p+q} h^{p,q}(Z),
   $$
where $h^{p,q}(Z)$ is the usual Hodge number of $Z$. In this case,
  $$
  e(Z)(u,v)= \sum_{p,q} h^{p,q}(Z) u^p v^q\,
  $$
is the (usual) Hodge polynomial of $Z$. Note that in this case,
the Poincar\'{e} polynomial of $Z$ is
  \begin{equation}\label{eqn:Poinca}
  P_Z(t)=\sum_k b^k(Z) t^k= \sum_k \left(
  \sum_{p+q=k} h^{p,q}(Z) \right) t^k= e(Z)(t,t).
  \end{equation}
where $b^k(Z)$ is the $k$-th Betti number of $Z$.

\begin{theorem}[{\cite[Theorem 2.2]{MOV1}\cite{Du}}] \label{thm:Du}
 Let $Z$ be a complex algebraic variety. Suppose that $Z$ is
 a finite disjoint union $Z=Z_1\cup \cdots \cup Z_n$, where the
 $Z_i$ are algebraic subvarieties. Then
  $$
  e(Z)= \sum_i e (Z_i).
  $$\hfill $\Box$
\end{theorem}

The following Hodge polynomials will be needed later:
\begin{enumerate}
 \item For the projective space $\PP^{n-1}$, we have
    $$
    e(\PP^{n-1})
    =1+uv+(uv)^2+\cdots +   (uv)^{n-1}=\frac{1-(uv)^{n}}{1-uv}\ .
    $$
 \item Let $\Jac^d X$ be the Jacobian of (any) degree $d$ of a
   (smooth, projective) complex curve $X$ of genus $g$. Then
    $$
    e(\Jac^d X)=(1+u)^g(1+v)^g.
    $$
 \item Let $X$ be a curve (smooth, projective) complex curve of genus
   $g$, and $k\geq 1$. The Hodge polynomial of the symmetric product
   $\Sym^k X$ is computed in \cite{Bur},
    $$
     e(\Sym^k X)=
    \coeff_{x^0}\frac{(1+ux)^{g}(1+vx)^{g}}{(1-x)(1-uvx)x^{k}}\, .
    $$
 \item The Hodge polynomial of the Grassmannian $\Gr(k,N)$ is given by
  \cite[Lemma 2.5]{MOV2},
    $$
    e(\Gr(k,N))= \frac{(1- (uv)^{N-k+1}) \cdots (1- (uv)^{N-1})
    (1- (uv)^{N})}{(1- uv)\cdots(1-(uv)^{k-1}) (1-(uv)^{k})}\, .
    $$
 \item Suppose that  $\pi:Z\to Y$ is an algebraic fiber bundle with
   fiber $F$ which is
   locally trivial in the Zariski topology, then \cite[Lemma 2.3]{MOV1}
     $$
     e(Z)=e(F)\,e(Y)\, .
     $$
   In particular this is true for $Z=F\times Y$.
 \item Suppose that $\pi:Z\to Y$ is a map between quasi-projective varieties which is
   a locally trivial fiber bundle in the usual topology, with fibers being projective
   spaces $F=\PP^N$ for some $N>0$. Then \cite[Lemma 2.4]{MOV2}
     $$
     e(Z)=e(F)\,e(Y)\,.
     $$
 \item Let $M$ be a smooth projective variety. Consider the algebraic variety
   $Z=(M\x M)/\ZZ_2$, where $\ZZ_2$ acts as $(x,y)\mapsto (y,x)$. The Hodge
   polynomial of $Z$ is \cite[Lemma 2.6]{MOV2}
     $$
     e(Z)= \frac12 \Big(e(M)(u,v)^2 + e(M)(-u^2,-v^2)\Big) \, .
     $$
\end{enumerate}

\section{Moduli spaces of triples}
\label{sec:stable-triples}

\subsection{Holomorphic triples}
\label{subsec:triples-definitions}

Let $X$ be a smooth projective curve of genus $g\geq 2$ over
$\CC$. A \emph{holomorphic triple} $T = (E_{1},E_{2},\phi)$ on $X$
consists of two holomorphic vector bundles $E_{1}$ and $E_{2}$
over $X$, of ranks $n_1$ and $n_2$ and degrees $d_1$ and $d_2$,
respectively, and a holomorphic map $\phi \colon E_{2} \to E_{1}$.
We refer to $(n_1,n_2,d_1,d_2)$ as the \emph{type} of $T$, to
$(n_1,n_2)$ as the \emph{rank} of $T$, and to $(d_1,d_2)$ as the
\emph{degree} of $T$.

A homomorphism from $T' = (E_1',E_2',\phi')$ to $T =
(E_1,E_2,\phi)$ is a commutative diagram
  \begin{displaymath}
  \begin{CD}
    E_2' @>\phi'>> E_1' \\
    @VVV @VVV  \\
    E_2 @>\phi>> E_1,
  \end{CD}
  \end{displaymath}
where the vertical arrows are holomorphic maps. A triple
$T'=(E_1',E_2',\phi')$ is a subtriple of $T = (E_1,E_2,\phi)$ if
$E_1'\subset E_1$ and $E_2'\subset E_2$ are subbundles,
$\phi(E_2')\subset E_1'$ and $\phi'=\phi|_{E_2'}$. A subtriple
$T'\subset T$ is called \emph{proper} if $T'\neq 0 $ and $T'\neq
T$. The quotient triple $T''=T/T'$ is given by $E_1''=E_1/E_1'$,
$E_2''=E_2/E_2'$ and $\phi'' \colon E_2''\to E_1''$ being the map
induced by $\phi$. We usually denote by $(n_1',n_2',d_1',d_2')$
and $(n_1'',n_2'',d_1'',d_2'')$, the types of the subtriple $T'$
and the quotient triple $T''$.

\begin{definition} \label{def:s-slope}
For any $\s \in \RR$ the \emph{$\s$-slope} of $T$ is defined by
 $$
   \mu_{\s}(T)  =
   \frac{d_1+d_2}{n_1+n_2} + \s \frac{n_{2}}{n_{1}+n_{2}}\ .
 $$
To shorten the notation, we define the \emph{$\mu$-slope} and
\emph{$\lambda$-slope} of the triple $T$ as $\mu=\mu(E_{1} \oplus
E_{2})= \frac{d_1+d_2}{n_1+n_2}$ and
$\lambda=\frac{n_{2}}{n_{1}+n_{2}}$, so that $\mu_{\s}(T)=\mu+\s
\lambda$.
\end{definition}

\begin{definition}\label{def:sigma-stable}
We say that a triple $T = (E_{1},E_{2},\phi)$ is
\emph{$\s$-stable} if
  $$
  \mu_{\s}(T') < \mu_{\s}(T) ,
  $$
for any proper subtriple $T' = (E_{1}',E_{2}',\phi')$. We define
\emph{$\s$-semistability} by replacing the above strict inequality
with a weak inequality. A triple is called \emph{$\s$-polystable}
if it is the direct sum of $\s$-stable triples of the same
$\s$-slope. It is \emph{$\s$-unstable} if it is not
$\s$-semistable, and \emph{properly $\s$-semistable} if it is
$\s$-semistable but not $\s$-stable. A $\s$-destabilizing
subtriple $T'\subset T$ is a proper subtriple satisfying
$\mu_{\s}(T') \geq \mu_{\s}(T)$.
\end{definition}

We denote by
  $$
  \cN_\s = \cN_\s(n_1,n_2,d_1,d_2)
  $$
the moduli space of $\s$-polystable triples $T =
(E_{1},E_{2},\phi)$ of type $(n_1,n_2,d_1,d_2)$, and drop the type
from the notation when it is clear from the context. The open
subset of $\s$-stable triples is denoted by $\cN_\s^s =
\cN_\s^s(n_1,n_2,d_1,d_2)$. This moduli space is constructed in
\cite{BGP} by using dimensional reduction. A direct construction
is given by Schmitt \cite{Sch} using geometric invariant theory.

There are certain necessary conditions in order for
$\s$-semistable triples to exist. Let $\mu_i=\mu(E_i)=d_i/n_i$
stand for the slope of $E_i$, for $i=1,2$. We write
  \begin{align*}
  \s_m = &\mu_1-\mu_2\ ,  \\
  \s_M = & \left(1+ \frac{n_1+n_2}{|n_1 - n_2|}\right)(\mu_1 - \mu_2)\ ,
      \qquad \mbox{if $n_1\neq n_2$\ .}
  \end{align*}

\begin{proposition}\cite{BGPG} \label{prop:alpha-range}
The moduli space $\cN_\s(n_1,n_2,d_1,d_2)$ is a complex projective
variety. Let $I$ denote the interval $I=[\s_m,\s_M]$ if $n_1\neq
n_2$, or $I=[\s_m,\infty)$ if $n_1=n_2$. A necessary condition for
$\cN_\s(n_1,n_2,d_1,d_2)$ to be non-empty is that $\s\in I$.
\hfill $\Box$
\end{proposition}

\bigskip

The moduli space $\cN_\s$ for the smallest possible values of the
parameter can be explicitly described. We refer to the value of
$\s$ given by $\s=\smp=\s_m+\epsilon$ ($\epsilon>0$ small) as
\textit{small}.

\begin{proposition}[{\cite[Proposition 4.10]{MOV1}}] \label{prop:moduli-small}
There is a map
 $$
 \pi:\cN_{\smp}=\cN_{\smp}(n_1,n_2,d_1,d_2) \to M(n_1,d_1) \times M(n_2,d_2)
 $$
which sends $T=(E_1,E_2,\phi)$ to $(E_1,E_2)$. If
$\GCD(n_1,d_1)=1$, $\GCD(n_2,d_2)=1$ and $\mu_1-\mu_2>2g-2$, then
$\cN_{\smp}$ is a projective bundle over $M(n_1,d_1) \times
M(n_2,d_2)$, whose fibers are projective spaces of dimension
$n_2d_1-n_1d_2- n_1n_2(g-1)-1$.
\hfill $\Box$
\end{proposition}

To study the dependence of the moduli spaces $\cN_\s$ on the
parameter, we need to introduce the concept of critical value
\cite{BGP,MOV1}.

\begin{definition}\label{def:critical}
The values of $\s_c\in I$ for which there exist $0 \le n'_1 \leq
n_1$, $0 \le n'_2 \leq n_2$, $d'_1$ and $d'_2$, with $n_1'n_2\neq
n_1n_2'$, such that
 \begin{equation}\label{eqn:sigmac}
 \s_c=\frac{(n_1+n_2)(d_1'+d_2')-(n_1'+n_2')(d_1+d_2)}{n_1'n_2-n_1n_2'},
 \end{equation}
are called \emph{critical values}.
\end{definition}

Given a triple $T=(E_1,E_2,\phi)$, the condition of
$\s$-(semi)stability for $T$ can only change when $\s$ crosses a
critical value. If $\s=\s_c$ as in (\ref{eqn:sigmac}) and if $T$
has a subtriple $T'\subset T$ of type $(n_1',n_2',d_1',d_2')$,
then $\mu_{\s_c}(T')=\mu_{\s_c}(T)$ and
 \begin{enumerate}
 \item if $\lambda'>\lambda$ (where $\lambda'$ is the $\lambda$-slope of
 $T'$), then $T$ is not $\s$-stable for
 $\s>\s_c$,
\item if $\lambda'<\lambda$, then $T$ is not $\s$-stable for
 $\s<\s_c$.
 \end{enumerate}
Note that $n_1'n_2\neq n_1n_2'$ is equivalent to $\lambda'\neq
\lambda$.

\begin{proposition}[{\cite[Proposition 2.6]{BGPG}}] \label{prop:triples-critical-range}
Fix $(n_1,n_2,d_1,d_2)$. Then
 \begin{enumerate}
 \item[(1)] The critical values are a finite number of values $\s_c \in I$.
 \item[(2)] The stability and semistability criteria  for two values of $\s$
  lying between two consecutive critical values are equivalent; thus
  the corresponding moduli spaces are isomorphic.
 \item[(3)] If $\s$ is not a critical value and $\GCD(n_1,n_2,d_1+d_2) = 1$,
  then $\s$-semistability is equivalent to $\s$-stability, i.e.,\
  $\cN_\s=\cN_\s^s$.
 \end{enumerate} \hfill $\Box$
\end{proposition}

\subsection{Extensions and deformations of triples}
\label{subsec:extensions-of-triples}

The homological algebra of triples is controlled by the
hypercohomology of a certain complex of sheaves which appears when
studying infinitesimal deformations \cite[Section 3]{BGPG}. Let
$T'=(E'_1,E'_2,\phi')$ and $T''=(E''_1,E''_2,\phi'')$ be two
triples of types $(n_{1}',n_{2}',d_{1}',d_{2}')$ and
$(n_{1}'',n_{2}'',d_{1}'',d_{2}'')$, respectively. Let
$\Hom(T'',T')$ denote the linear space of homomorphisms from $T''$
to $T'$, and let $\Ext^1(T'',T')$  denote the linear space of
equivalence classes of extensions of the form
 $$
  0 \lto T' \lto T \lto T'' \lto 0,
 $$
where by this we mean a commutative  diagram
  $$
  \begin{CD}
  0@>>>E_1'@>>>E_1@>>> E_1''@>>>0\\
  @.@A\phi' AA@A \phi AA@A \phi'' AA\\
  0@>>>E'_2@>>>E_2@>>>E_2''@>>>0.
  \end{CD}
  $$
To analyze $\Ext^1(T'',T')$ one considers the complex of sheaves
 \begin{equation} \label{eqn:extension-complex}
    C^{\bullet}(T'',T') \colon ({E_{1}''}^{*} \otimes E_{1}') \oplus
  ({E_{2}''}^{*} \otimes E_{2}')
  \overset{c}{\lto}
  {E_{2}''}^{*} \otimes E_{1}',
 \end{equation}
where the map $c$ is defined by
 $$
 c(\psi_{1},\psi_{2}) = \phi'\psi_{2} - \psi_{1}\phi''.
 $$

We introduce the following notation:
\begin{align*}
  \HH^i(T'',T') &= \HH^i(C^{\bullet}(T'',T')), \\
  h^{i}(T'',T') &= \dim\HH^{i}(T'',T'), \\ 
  \chi(T'',T') &= h^0(T'',T') - h^1(T'',T') + h^2(T'',T'). 
\end{align*}

\begin{proposition}[{\cite[Proposition 3.1]{BGPG}}]
  \label{prop:hyper-equals-hom}
  There are natural isomorphisms
  \begin{align*}
    \Hom(T'',T') &\cong \HH^{0}(T'',T'), \\
    \Ext^{1}(T'',T') &\cong \HH^{1}(T'',T'),
  \end{align*}
and a long exact sequence associated to the complex
$C^{\bullet}(T'',T')$:
 $$
 \begin{array}{c@{\,}c@{\,}c@{\,}l@{\,}c@{\,}c@{\,}c}
  0 &\lto \mathbb{H}^0(T'',T') &
  \lto & H^0(({E_{1}''}^{*} \otimes E_{1}') \oplus ({E_{2}''}^{*} \otimes
  E_{2}'))
  & \lto &  H^0({E_{2}''}^{*} \otimes E_{1}') \\
    &  \lto \mathbb{H}^1(T'',T') &
  \lto &  H^1(({E_{1}''}^{*} \otimes E_{1}') \oplus ({E_{2}''}^{*} \otimes
  E_{2}'))
 &  \lto & H^1({E_{2}''}^{*} \otimes E_{1}') \\
 &   \lto \mathbb{H}^2(T'',T') & \lto & 0. & &
 \end{array}
 $$ \hfill $\Box$
\end{proposition}

\begin{proposition}[{\cite[Proposition 3.2]{BGPG}}]
  \label{prop:chi(T'',T')}
  For any holomorphic triples $T'$ and $T''$ we have
  \begin{align*}
    \chi(T'',T') &= \chi({E_{1}''}^{*} \otimes E_{1}')
    + \chi({E_{2}''}^{*} \otimes E_{2}')
    - \chi({E_{2}''}^{*} \otimes E_{1}')  \\
    &= (1-g)(n''_1 n'_1 + n''_2 n'_2 - n''_2 n'_1) + n''_1 d'_1 - n'_1 d''_1
    + n''_2 d'_2 - n'_2 d''_2
    - n''_2 d'_1 + n'_1 d''_2,
  \end{align*}
where $\chi(E)=\dim H^0(E) - \dim H^1(E)$ is the Euler
characteristic of $E$. \hfill $\Box$
\end{proposition}

\begin{lemma}[{\cite[Proposition 3.5]{BGPG}}] \label{lem:h0-vanishing}
   Suppose that $T'$ and $T''$ are $\s$-semistable, for some
   value of $\s$.
 \begin{enumerate}
  \item[(1)] If $\mu_\s(T')<\mu_\s (T'')$ then
  $\HH^{0}(T'',T') = 0$.
  \item[(2)] If $\mu_\s(T')=\mu_\s (T'')$ and $T'$, $T''$ are
   $\s$-stable, then
  $$
     \HH^{0}(T'',T') \cong
     \begin{cases}
       \CC \quad &\text{if $T' \cong T''$} \\
       0 \quad &\text{if $T' \not\cong T''$}.
     \end{cases}
  $$
 \end{enumerate} \hfill $\Box$
\end{lemma}

\begin{lemma} \label{lem:H2=0}
 If $T''=(E_1'',E_2'',\phi'')$ is an injective triple, that is $\phi'':E_2''\to E_1''$ is
 injective, then $\HH^2(T'',T')=0$.
\end{lemma}

\begin{proof}
 Since $E_2''\to E_1''$ is injective, we have that $E_2''\ox
 {E_1'}^{*} \ox K \to E_1''\ox
 {E_1'}^{*} \ox K$ is injective as well. Therefore $H^0(E_2''\ox
 {E_1'}^{*} \ox K) \to H^0(E_1''\ox
 {E_1'}^{*} \ox K)$ is a monomorphism. Taking duals, $H^1({E_{1}''}^{*} \otimes
  E_{1}') \to H^1({E_{2}''}^{*} \otimes E_{1}')$ is an
  epimorphism. Proposition \ref{prop:hyper-equals-hom} implies
  that $\HH^2(T'',T')=0$.
\end{proof}

Since the  space of infinitesimal deformations of a triple $T$ is
isomorphic to $\HH^{1}(T,T)$, the previous results also apply to
studying deformations of a holomorphic triple $T$.

\begin{theorem}[{\cite[Theorem 3.8]{BGPG}}]\label{thm:smoothdim}
Let $T=(E_1,E_2,\phi)$ be an $\s$-stable triple of type
$(n_1,n_2,d_1,d_2)$.
 \begin{enumerate}
 \item[(1)] The Zariski tangent space at the point defined by $T$
 in the moduli space of stable triples  is isomorphic to
 $\HH^{1}(T,T)$.
 \item[(2)] If\/ $\HH^{2}(T,T)= 0$, then the moduli space of
 $\s$-stable triples is smooth in  a neighbourhood of the point
 defined by $T$.
 \item[(3)] At a smooth point $T\in \cN^s_\s(n_1,n_2,d_1,d_2)$ the
 dimension of the moduli space of $\s$-stable triples is
 \begin{align*}
  \dim \cN^s_\s(n_1,n_2,d_1,d_2)
  &= h^{1}(T,T) = 1 - \chi(T,T) \\
  &= (g-1)(n_1^2 + n_2^2 - n_1 n_2) - n_1 d_2 + n_2 d_1 + 1.
 \end{align*}
 \item[(4)] Let $T=(E_1,E_2,\phi)$ be a $\s$-stable triple. If $T$
 is an injective triple, then the moduli space is smooth at $T$.
 \end{enumerate} \hfill $\Box$
\end{theorem}

\section{Description of the flip loci}
\label{sec:flip-loci}

Fix the type $(n_1,n_2,d_1,d_2)$ for the moduli spaces of
holomorphic triples. We want to describe the differences between
two spaces $\cN^s_{\s_1}$ and $\cN^s_{\s_2}$ when $\s_1$ and
$\s_2$ are separated by a critical value. Let $\s_c\in I$ be a
critical value and set
 $$
 \scp = \s_c + \epsilon,\quad \scm = \s_c -
 \epsilon,
 $$
where $\epsilon > 0$ is small enough so that $\s_c$ is the only
critical value in the interval $(\scm,\scp)$.

\begin{definition}\label{def:flip-loci}
We define the \textit{flip loci} as
 \begin{align*}
 \cS_{\scp} &= \{ T\in\cN_{\scp} \ |
 \ \text{$T$ is $\scm$-unstable}\} \subset\cN_{\scp} \ ,\\
 \cS_{\scm} &= \{ T\in\cN_{\scm} \ |
 \ \text{$T$ is $\scp$-unstable}\}
 \subset\cN_{\scm} \ .
 \end{align*}
 Note that $\cN_{\scp}-\cS_{\scp}=\cN_{\scm}-\cS_{\scm}$.
\end{definition}

For $\s_c=\s_m$, $\cN_{\s_m^-}$ is empty, hence $\cN_{\smp}=
\cS_{\smp}$.  Analogously, when $n_1\neq n_2$, $\cN_{\s_M^+}$ is
empty and $\cN_{\s_M^-}= \cS_{\s_M^-}$.

We shall describe geometrically the flip loci $\cS_{\s_c^{\pm}}$, for
a critical value $\s_c$, in the situations which suffice for our
purposes. In order to do this, recall that for a properly
$\s_c$-semistable triple $T$, there is a Jordan-H\"older
filtration
 \begin{equation}\label{eqn:JH}
 0\subset T_1\subset T_2\subset \cdots \subset T_r=T\, ,
 \end{equation}
where $\bar{T}_i=T_i/T_{i-1}$, $i=1,\ldots,r$, are $\s_c$-stable
triples. Although the filtration (\ref{eqn:JH}) is not uniquely
defined, the graded triple associated to $T$,
 \begin{equation}\label{eqn:graded}
 \mathrm{gr}_{\s_c} (T)= \bigoplus_{i=1}^r \bar{T}_i
 \end{equation}
is well-defined, up to order of the summands. To describe
$\cS_{\s_c^{\pm}}$, we shall stratify them according to the different
possibilities for (\ref{eqn:graded}). We do this in the cases
$r\leq 3$.

The case where $r=2$ is specially simple.

\begin{proposition}\label{prop:locus1}
 Let $\s_c$ be a critical value given by $(n_1',n_2',d_1',d_2')$ in
 (\ref{eqn:sigmac}), and let $(n_1'',n_2'',d_1'',d_2'')=(n_1-n_1',n_2-n_2',d_1-d_1',d_2-d_2')$.
 Let $X^+\subset \cS_\scp$ (resp.\ $X^-\subset \cS_\scm$)
 be the subset of those
 triples $T$ such that $T$ sits in a non-split exact sequence
  \begin{equation}\label{eqn:exact}
  0\to T'\to T\to T''\to 0
  \end{equation}
 where $T'\in \cN_{\s_c}'=\cN_\s(n_1',n_2',d_1',d_2')$
 and $T''\in \cN_{\s_c}''=\cN_\s(n_1'',n_2'',d_1'',d_2'')$ are both
 $\s_c$-stable, and $\lambda'<\lambda$ (resp.\ $\lambda'>\lambda$).
 Assume that $\HH^2(T'',T')=0$, for every $T'\in \cN_{\s_c}'^{,s}$, $T''\in
 \cN_{\s_c}''^{,s}$.
 Then $X^+$ (resp.\ $X^-$) is the projectivization of a bundle of
 rank $-\chi(T'',T')$ over $\cN_{\s_c}'^{,s}\x \cN_{\s_c}''^{,s}$.
\end{proposition}

\begin{proof}
Let us do the case of $X^+$, the other one being analogous. First
of all, note that for any $T\in X^+$, there is a unique exact
sequence like (\ref{eqn:exact}) in which $T$ sits. For consider
any proper non-trivial $\tilde{T}\subset T$ with the same
$\s_c$-slope as $T$, compose with the projection $T\surj T''$ to
get a map $\tilde{T}\to T''$ between triples of the same
$\s_c$-slope. As $T''$ is $\s_c$-stable, either this map is zero
or an epimorphism. In the first case, $\tilde{T}\subset T'$ and
both are non-zero triples of the same $\s_c$-slope, so
$\tilde{T}=T'$ by $\s_c$-stability of $T'$. In the second case, we
have a short exact sequence $0\to \tilde{T}'\to \tilde{T}\to
T''\to 0$, where $\tilde{T}'\subset T'$ has the same $\s_c$-slope
as $T',T''$ and $\tilde{T}$. Therefore $\tilde{T}'=0$ (since
$\tilde{T}$ is properly contained in $T$), and hence
$\tilde{T}\cong T''$. This gives a splitting of the exact sequence
(\ref{eqn:exact}), contrary to our assumption.

The above implies that $X^+$ is parametrized by the extensions
(\ref{eqn:exact}). Now as $T'\in \cN_{\s_c}'$ and $T''\in
\cN_{\s_c}''$ are both $\s_c$-stable and non-isomorphic (for
instance, because $\lambda'\neq \lambda''$), Lemma
\ref{lem:h0-vanishing} implies that $\HH^0(T'',T')=0$. Then
$\HH^1(T'',T')$ has constant dimension equal to
  $$
  \dim \HH^1(T'',T')= -\chi(T'',T')\, .
  $$
Therefore the extensions give a vector bundle of rank
$-\chi(T'',T')$ over  $\cN_{\s_c}'^{,s}\x \cN_{\s_c}''^{,s}$.

If two extensions give rise to the same triple $T$, then the
uniqueness of the subtriple $T'$ yields the existence of a diagram
 $$
  \begin{CD}
  0@>>>T'@>>>T@>>> T''@>>>0\\
  @.@V VV@V \cong VV@V  VV\\
  0@>>>T'@>>>T@>>>T''@>>>0,
  \end{CD}
  $$
where the left and right vertical arrows are automorphisms of $T'$
and $T''$ respectively. Since both of them are $\s_c$-stable, we
have that $\Aut(T')=\Aut(T'')=\CC^*$. 
The action of $\Aut(T')\x \Aut(T'')$ on $\Ext^1(T'',T')-\{0\}$
factors through the action of $\CC^*$ by multiplication on the
fibers $\HH^1(T'',T')-\{0\}$. So $X^+\to\cN_{\s_c}'^{,s}\x
\cN_{\s_c}''^{,s}$ is a projective bundle with fibers
$\PP\HH^1(T'',T')$.
\end{proof}

To study the flips in the case of triples of rank $(3,1)$, we
shall need to deal also with $\s_c$-semistable triples $T$ such
that (\ref{eqn:graded}) has $r=3$ terms. From now on, assume that
$\GCD(n_1,n_2,d_1+d_2)=1$, so that $\cN_{\s}^s=\cN_{\s}$ for
non-critical values $\s$.

In the following, fix a critical value $\s_c$ given by
$(n_1',n_2',d_1',d_2')$ in (\ref{eqn:sigmac}), and let
$(n_1'',n_2'',d_1'',d_2'')=(n_1-n_1',n_2-n_2',d_1-d_1',d_2-d_2')$.
Suppose that $T\in \cS_{\s_c^+}$ (resp.\ $T\in \cS_{\s_c^-}$)  is
a properly $\s_c$-semistable triple with $r=3$ terms in the
Jordan-H\"older filtration. Then $T$ sits in a non-split exact
sequence
  \begin{equation}\label{eqn:uno}
  0\to T'\to T\to T''\to 0
  \end{equation}
where $T'$ is $\s_c$-semistable, $T''$ is $\s_c$-stable and
$\lambda'<\lambda$ (resp.\ $T'$ is $\s_c$-stable, $T''$ is
$\s_c$-semistable and $\lambda'>\lambda$). For this it is enough
to take $T'=T_{r-1}$ (resp.\ $T=T_1$) in the Jordan-H\"older
filtration (\ref{eqn:JH}).

Moreover, there is an exact sequence
  \begin{equation}\label{eqn:dos}
  0\to T_1\to T'\to T_2\to 0 \qquad \text{(resp.\ }   0\to T_1\to T''\to T_2\to 0
  \text{)}
  \end{equation}
where $T_1,T_2$ are $\s_c$-stable triples of the same
$\s_c$-slope. The sequence (\ref{eqn:dos}) may be split or
non-split. Note that the graded triple associated to $T$ is
$T''\oplus T_1\oplus T_2$ (resp.\ $T'\oplus T_1\oplus T_2$).

Finally, we shall denote by $\cN_{\s_c}^1$ and $\cN_{\s_c}^2$ the
moduli spaces of triples of the types determined by $T_1,T_2$,
respectively. Let $\lambda_1,\lambda_2$ denote their
$\lambda$-slopes. Notice that $\lambda_1<\lambda$ (resp.\
$\lambda_2 < \lambda$).

We stratify $\cS_{\s_c^\pm}$ according to whether (\ref{eqn:dos})
is split or non-split, and to whether $T_1\cong T_2$ or
$T_1\not\cong T_2$.

\begin{proposition}\label{prop:locus2}
In the above situation, let $X^+\subset \cS_\scp$ (resp.\
$X^-\subset \cS_\scm$) be the subset of those triples $T$ which
sit in a non-split exact sequence (\ref{eqn:uno}), where the exact
sequence (\ref{eqn:dos}) is non-split and $T_1\not\cong T_2$.

 Let $U\subset \cN_{\s_c}^{1,s}\x\cN_{\s_c}^{2,s}$ be the open set
 consisting of those $(T_1,T_2)$ with $T_1\not\cong T_2$ (which is
 the whole space in case the types are different).
 Assume that $\lambda_2<\lambda$ and
 $\HH^2(T'',T_1)=\HH^2(T'',T_1)=\HH^2(T_2,T_1)=0$, for every $T''\in
 \cN_{\s_c}''^{,s}$, $(T_1,T_2)\in U$
 (resp.\ $\lambda_1<\lambda$ and
 $\HH^2(T_1,T')=\HH^2(T_2,T')=\HH^2(T_2,T_1)=0$, for every $T'\in
 \cN_{\s_c}'^{,s}$, $(T_1,T_2)\in U$). Then
 \begin{enumerate}
 \item The space $Y$ parametrizing the triples $T'$ (resp.\ $T''$) is
 the projectivization of a bundle of rank $-\chi(T_2,T_1)$ over $U$.
 \item $X^+$ (resp.\ $X^-$) is a bundle over $Y\x \cN_{\s_c}''^{,s}$
 (resp.\ $\cN_{\s_c}'^{,s}\x Y$) with
 fibers $\PP^{a-1} - \PP^{b-1}$, $a=-\chi(T'',T')$, $b=-\chi(T'',T_1)$
 (resp.\ $a=-\chi(T'',T')$, $b=-\chi(T_2,T')$).
 \end{enumerate}
\end{proposition}

\begin{proof}
We shall do the case of $X^-$, the other being analogous. By
assumption $T_1\not\cong T_2$, hence $\HH^0(T_1,T_2)=0$, by Lemma
\ref{lem:h0-vanishing}. As $\HH^2(T_2,T_1)=0$, we have that $\dim
\Ext^1(T_2,T_1)=-\chi(T_2,T_1)$. Hence the space parametrizing
triples $T''$ as in (\ref{eqn:dos}) is the projectivization of a
bundle of rank $-\chi(T_2,T_1)$ over $U$.

Now, for each $(T',T'')\in \cN_{\s_c}'^{,s}\x Y$, we have an exact
sequence,
    $$
    \begin{aligned}
    0 \to \ & \Hom(T_2,T') \to \Hom(T'',T') \to \Hom(T_1,T') \to
    \\ &
        \Ext^1(T_2,T') \to \Ext^1(T'',T') \to \Ext^1(T_1,T') \to \\ &
        \HH^2(T_2,T') \to \HH^2(T'',T') \to \HH^2(T_1,T') \to 0\,
        .
    \end{aligned}
    $$
As $T_1 \not\cong T'$ and $T_2 \not\cong T'$ (note that
$\lambda'>\lambda>\lambda_1,\lambda_2$), we have that
$\HH^0(T_1,T')=\HH^0(T_2,T')=0$. By assumption
$\HH^2(T_1,T')=\HH^2(T_2,T')=0$. Therefore
$\HH^0(T'',T')=\HH^2(T'',T')=0$, and hence $\dim
\HH^1(T'',T')=-\chi(T'',T')$. The space parametrizing extensions
as in (\ref{eqn:uno}) is a bundle over $\cN_{\s_c}'^{,s}\x Y$
whose fibers are $\Ext^1(T'',T')$. However, not all of them give
that $T$ is  $\scp$-stable. For this it is necessary (and
sufficient) that $T$ does not have a subtriple $\tilde T \subset
T$ with the same $\s_c$-slope and $\tilde\lambda \leq \lambda$.
Such $\tilde T$ projects to a subtriple of $T''$ of the same
$\s_c$-slope. But $T''$ only has one sub-triple of the same
$\s_c$-slope, namely $T_1$. If $\tilde T$ contains $T'$, then it
is the kernel of $T\surj T'' \surj T_2$, and hence $\tilde\lambda
> \lambda$ (since $\lambda_2<\lambda$).
If not, we must have $\tilde{T}\cong T_1$. Therefore the class of
the extension $\xi \in \Ext^1(T'',T')$ defining $T$ lies in the
kernel of $\Ext^1(T'',T')\to \Ext^1(T_1,T')$, hence in the image
of
      $$
      \Ext^1(T_2,T')\subset \Ext^1(T'',T') \, .
      $$
As a conclusion $T$ is $\scp$-stable if and only if it is defined
by an extension in  $A=\Ext^1(T'',T') - \Ext^1(T_2,T')$.

A triple $T\in \cS_\scp$ determines uniquely $T'$ as the only
$\s_c$-stable subtriple with $\lambda'>\lambda$. Hence we must
quotient $A$ by the action of $\Aut(T')\x \Aut(T'')$. However,
$T'$ is $\s_c$-stable, hence $\Aut(T') = \CC^*$. Also
$\Aut(T'')=\CC^*$, since it is a non-trivial extension of two
non-isomorphic $\s_c$-stable triples. Therefore the action
consists on multiplication by scalars in $\Ext^1(T'',T')$. The
result follows.
\end{proof}

\begin{proposition}\label{prop:locus3}
In the above situation, let $X^+\subset \cS_\scp$
 (resp.\ $X^-\subset \cS_\scm$)
 be the subset of those
 triples $T$ such that $T$ sits in a non-split exact sequence
 (\ref{eqn:uno}), where the
 exact sequence (\ref{eqn:dos}) is non-split and $T_1 \cong T_2$.
 Assume that
 $\HH^2(T'',T_1)=\HH^2(T_1,T_1)=0$, for every $T_1\in
\cN_{\s_c}^{1,s}$,
 $T''\in \cN_{\s_c}''^{,s}$ (resp.\
 $\HH^2(T_1,T')=\HH^2(T_1,T_1)=0$, for every $T_1\in \cN_{\s_c}^{1,s}$,
 $T'\in \cN_{\s_c}'^{,s}$).
 Then
 \begin{enumerate}
 \item Then space
 $Y$ parametrizing the triples $T'$ (resp.\ $T''$) is the projectivization of
 a bundle of rank $-\chi(T_1,T_1)+1$ over $\cN_{\s_c}^{1,s}$.
 \item $X^+$ (resp.\ $X^-$) is a $\CC^{a-1}$-bundle over a
 $\PP^{a-1}$-bundle over $Y\x \cN_{\s_c}''^{,s}$ (resp.\
 $\cN_{\s_c}'^{,s}\x Y$), where
 $a=-\chi(T'',T_1)$ (resp.\ $a=-\chi(T_1,T')$).
 \end{enumerate}
\end{proposition}

\begin{proof}
  We do the case of $X^-$, the other being similar.
  It is clear that $\HH^0(T_1,T_1)=\CC$, since $T_1$ is
  $\s_c$-stable. Therefore, $\dim\HH^1(T_1,T_1)
  =-\chi(T_1,T_1)+1$, and the space $Y$ is a bundle over $\cN_{\s_c}^{1,s}$
  whose fibers are the projective spaces $\PP\HH^1(T_1,T_1)$.

  Fix $(T',T'')\in\cN_{\s_c}'^{,s}\x Y$. Then the space of
  extensions $\Ext^1(T'',T')$ sits in a short exact sequence
   \begin{equation}\label{eqn:spli}
   0\to \Ext^1(T_1,T')\to \Ext^1(T'',T')\to \Ext^1(T_1,T')\to 0 \,
   .
   \end{equation}
  As in the proof of Proposition \ref{prop:locus2}, $T$ is
  $\scp$-stable if and only if the extension class lies in
  $A=\Ext^1(T'',T')- \Ext^1(T_1,T')$.

  From (\ref{eqn:spli}), there is a (non canonical) isomorphism $\Ext^1(T'',T')
  \cong \Ext^1(T_1,T')\oplus\Ext^1(T_1,T')$, so that
  $A$ gets identified to $\Ext^1(T_1,T') \x (\Ext^1(T_1,T')-\{ 0 \})$.  Let us find which
  extensions give rise to the same (isomorphism class of) triple $T$. As $T$ uniquely
  determines $T'$ and $T''$, we must look at the action of $\Aut(T')\times \Aut(T'')$ on
  $\Ext^1(T'',T')$. Again $\Aut(T')=\CC^*$. On the other hand,
  $0\to T_1\to T''\to T_1 \to
  0$ gives that
     $$
     \Aut(T'') = \CC \x\CC^*\, ,
     $$
  where $(\lambda,\mu)\in \CC \x\CC^*$ acts as
  $(x,y) \mapsto ( \mu\, x +\lambda\, y  ,\mu\, y)$ (for this, choose a $\cC^\infty$ splitting
  $T''\cong T_1\oplus T_1$; the splitting is non canonical, but the above formula
  is independent of the splitting). Therefore, there is an induced
  action
  $$
  (\lambda,\mu) \, : \, (u,v) \mapsto ( \mu\, u +\lambda\, v  ,\mu\, v)
  $$
  on $A$. The quotient of $A$ by this action is isomorphic to
     $$
     (\CC^a \x (\CC^a -\{0\} ) ) / \CC \x\CC^* \, .
     $$
Projecting onto the second factor gives a fiber bundle over
$\PP^{a-1}$ whose fibers are $\CC^{a-1}$.
\end{proof}

\begin{proposition}\label{prop:locus4}
In the above situation, let $X^+\subset \cS_\scp$
 (resp.\ $X^-\subset \cS_\scm$)
 be the subset of those
 triples $T$ such that $T$ sits in a non-split exact sequence
 (\ref{eqn:uno}), where the
 exact sequence (\ref{eqn:dos}) is split and $T_1 \not \cong T_2$
 (note that it must be $\lambda_1,\lambda_2<\lambda$ in this case).

 Assume $\HH^2(T'',T_1)=\HH^2(T'',T_2)=0$, for every $T_1,T_2,T''$
 (resp.\ $\HH^2(T_1,T')=\HH^2(T_2,T')=0$, for every $T_1,T_2,T'$).
 Then
 \begin{enumerate}
  \item The space $Y$ parametrizing the triples $T'$ (resp.\ $T''$) is either
  $\cN_{\s_c}^{1,s} \x\cN_{\s_c}^{2,s}$, if $T_1,T_2$ have
  different types, or
  $(\cN_{\s_c}^{1,s} \x\cN_{\s_c}^{1,s}-\Delta)/\ZZ_2$, if $T_1,T_2$
  are of the same type, where $\Delta$ is the diagonal, and
  $\ZZ_2$ acts by permutations.

  \item $X^+$ (resp.\ $X^-$) is a bundle over $Y\x \cN_{\s_c}''^{,s}$
  (resp.\ over $\cN_{\s_c}'^{,s}\x Y$) with fibers
  $\PP^{a-1}\x \PP^{b-1}$, where $a=-\chi(T'',T_1)$, $b=-\chi(T'',T_2)$
  (resp.\ $a=-\chi(T_1,T')$, $b=-\chi(T_2,T')$). This fiber bundle is
  Zariski locally trivial in the first case. In the second
  case,
   $$
   X^\pm= (Z\x_S Z-(p\x_S p)^{-1}(\Delta))/\ZZ_2,
   $$
  where $S=\cN_{\s_c}''^{,s}$ (resp.\
  $S=\cN_{\s_c}'^{,s}$), and $p:Z\to S\x
  \cN_{\s_c}^{1,s}$ is a Zariski locally trivial $\PP^{a-1}$-bundle, and $\ZZ_2$ acts by permutations.
 \end{enumerate}
\end{proposition}

\begin{proof}
Again we deal with the case of $X^-$. The first statement is
clear. The extensions parametrizing $T$ are
       $$
      \Ext^1(T'',T')= \Ext^1(T_1,T') \oplus \Ext^1(T_2,T') \cong \CC^a \oplus \CC^b\, .
       $$
By the $\s_c$-stability, $\HH^0(T'',T')= \HH^0(T_1,T') \oplus
\HH^0(T_2,T')=0$. By assumption, $\HH^2(T'',T')= \HH^2(T_1,T')
\oplus \HH^2(T_2,T')=0$, so $a=\dim \HH^1(T_1,T')=-\chi(T_1,T')$
and  $b=\dim \HH^1(T_2,T')=-\chi(T_2,T')$.

For $T$ to be $\scp$-stable, it is necessary and sufficient that
it does not contain a subtriple $\tilde T \subset T$ of the same
$\s_c$-slope as $T$, and with $\tilde\lambda <\lambda$. Therefore,
$\tilde T\cong T_1$ or $T_2$. If there is an inclusion $T_i=\tilde
T\inc T$, then the extension class defining $T$ lies in the image
of
      $$
      \Ext^1(T_j,T')\subset \Ext^1(T'',T')
      $$
for $j=3-i$ ($i=1,2$). Therefore the $\scp$-stable triples are
defined by extensions in
      \begin{equation}\label{eqn:lll}
       \Ext^1(T'',T') -( \Ext^1(T_1,T' )\cup \Ext^1(T_2,T' )) =
       (\Ext^1(T_1,T' )-\{0\})\x (\Ext^1(T_2,T' )-\{0\})
      \end{equation}

Let us find which extensions give rise to the same (isomorphism
class of) $T$. Any automorphism of $T$ induces automorphisms of
$T'$ and $T''$. Clearly, $\Aut(T')=\CC^*$. On the other hand,
  $$
  \Aut(T'')= \CC^*\x \CC^*
  $$
acting diagonally on both factors of $\Ext^1(T'',T')=
\Ext^1(T_1,T') \oplus \Ext^1(T_2,T')$ (note that $T_1, T_2$ are
non-isomorphic). If we quotient (\ref{eqn:lll}) by this action, we
get a fiber bundle over $Y$ with fibers
 $$
 (\Ext^1(T_1,T' )-\{0\})\x (\Ext^1(T_2,T' )-\{0\})/\CC^*\x \CC^*
 =\PP \Ext^1(T_1,T' ) \x \PP\Ext^1(T_2,T' )\, .
 $$

The final assertion is clear, since $Z$ is the fiber bundle with
fiber $\PP \Ext^1(T_1,T')$ over $(T',T_1)\in S\x
\cN_{\s_c}^{1,s}$, which is an algebraic vector bundle.
\end{proof}

\begin{proposition}\label{prop:locus5}
In the above situation, let $X^+\subset \cS_\scp$
 (resp.\ $X^-\subset \cS_\scm$)
 be the subset of those
 triples $T$ such that $T$ sits in a non-split exact sequence
 (\ref{eqn:uno}), where the
 exact sequence (\ref{eqn:dos}) is split and $T_1 \cong T_2$.

 Assume that $\HH^2(T'',T_1)=0$, for every $T_1,T''$ (resp.\
 $\HH^2(T_1,T')=0$, for every $T_1,T'$).
 Then $X^+$ (resp.\ $X^-$) is a bundle
 with fibers the grassmannian $\Gr(2, a)$, $a=-\chi(T'',T_1)$ (resp.\
 $a=-\chi(T_1,T')$),
 over $\cN_{\s_c}''^{,s} \x \cN_{\s_c}^{1,s}$
 (resp.\ over $\cN_{\s_c}'^{,s} \x \cN_{\s_c}^{1,s}$).
\end{proposition}

\begin{proof}
   We treat the case of $X^-$.
   The triples $T''=T_1\oplus T_1$ are parametrized by $\cN_{\s_c}^{1,s}$.
   Now for any $(T',T_1)\in \cN_{\s_c}'^{,s}\x \cN_{\s_c}^{1,s}$,
   the extensions (\ref{eqn:uno}) are parametrized by
      \begin{equation}\label{eqn:extens}
      \Ext^1(T'',T')= \Ext^1(T_1,T') \oplus \Ext^1(T_1,T') =\Ext^1(T_1,T')\ox \CC^2\, .
      \end{equation}
   An extension gives rise to a
    $\scp$-unstable triple $T$ if there is a subtriple $\tilde{T}\subset T$ of the same
   $\s_c$-slope with $\tilde\lambda< \lambda$.
   The only possibility is that $\tilde{T}\cong T_1$. Composing with the projection $T\surj T''$,
   we get an embedding $\imath: T_1\inc T''=T_1\oplus T_1$. So there exists $(a,b)\in \CC^2-\{0\}$ such
   that $\imath(x)=(ax,bx)$. The quotient $T''/\imath(\tilde
   T)$ is isomorphic to $T_1$, and the extension is in the image of
   $$
    \Ext^1(T_1,T') \subset \Ext^1(T_1,T')\ox \CC^2\, ,
   $$
   embedded via $\xi\mapsto (a\xi,b\xi)$. Therefore the $\scp$-stable triples
   correspond to
    $$
    A=\{ (\xi_1,\xi_2) \in \Ext^1(T_1,T') \oplus \Ext^1(T_1,T') \, |\,
    \text{$\xi_1,\xi_2$ linearly independent} \}.
    $$
    To find $X^-$, we must quotient by the action of $\Aut(T')=\CC^*$ and
   $\Aut(T'')=\GL(2,\CC)$ on the space of extensions (\ref{eqn:extens}). This yields
   an action of $\GL(2,\CC)$ on $A$.
   The quotient is the grassmannian $\Gr(2, \Ext^1(T_1,T'))$. The
   result follows.
\end{proof}

\section{Hodge polynomials of the moduli spaces of
triples of ranks $(2,1)$} \label{sec:Poincare(2,1)}

In this section we recall the main results of \cite{MOV1}. Let
$\moduli=\moduli(2,1,d_1,d_2)$ denote the moduli space of
$\s$-polystable triples $T=(E_1,E_2,\phi)$ where $E_1$ is a vector
bundle of degree $d_1$ and rank $2$ and $E_2$ is a line bundle of
degree $d_2$. By Proposition \ref{prop:alpha-range}, $\s$ is in
the interval
  $$
    I=[\s_m,\s_M]=
    [\mu_1-\mu_2\,,\,4(\mu_1-\mu_2)]=[d_1/2-d_2,2d_1-4d_2], \qquad\mbox{ where }
    \mu_1-\mu_2\ge0\, .
  $$
Otherwise $\cN_\s$ is empty.

\begin{theorem}\label{thm:moduli(2,1)}
For $\s\in I$, $\moduli$ is a projective variety. It is smooth and
of (complex) dimension $3g-2 + d_1 - 2 d_2$ at the stable points
$\moduli^s$. Moreover, for non-critical values of $\s$,
$\moduli=\moduli^s$ (hence it is smooth and projective). \hfill
$\Box$
\end{theorem}

\begin{lemma}[{\cite[Lemma 5.3]{MOV1}}] \label{lem:dM}
The critical values for the moduli spaces of triples of type
$(2,1,d_1,d_2)$ are the numbers $\s_c=3d_M-d_1-d_2$ with
$\mu_1\leq d_M\leq d_1-d_2$. Furthermore, $\s_c=\s_m
\Leftrightarrow d_M=\mu_1$. \hfill $\Box$
\end{lemma}

The Hodge polynomials of the moduli spaces $\cN_\s$ for
non-critical values of $\s$ are computed in \cite{MOV1}.

\begin{theorem}[{\cite[Theorem 6.2]{MOV1}}] \label{thm:polinomiono(2,1)no-critico}
Suppose that $\s>\s_m$ is not a critical value. Set
$d_0=\Big[\frac13(\s+d_1+d_2)\Big]+1$. Then the Hodge polynomial
of $\moduli=\moduli (2,1,d_1,d_2)$ is
 $$
  e(\moduli)= \coeff_{x^0}
  \left[\frac{(1+u)^{2g}(1+v)^{2g}(1+xu)^{g}(1+xv)^{g}}{(1-uv)(1-x)(1-uvx)x^{d_1-d_2-d_0}}
    \Bigg(\frac{(uv)^{d_1-d_2-d_0}}{1-(uv)^{-1}x}-
    \frac{(uv)^{-d_1+g-1+2d_0}}{1-(uv)^2x}\Bigg)\right] .
 $$ \hfill $\Box$
\end{theorem}

We also can give the formula for the Hodge polynomial of the
moduli space of $\s$-stable triples when $\s>\s_m$ is a critical
value. Note however that this moduli space is smooth but
non-compact, so the Poincar\'{e} polynomial is not recovered from the
Hodge polynomial.

\begin{proposition} \label{prop:polinomio(2,1)critico}
Let $\s_c=3 \bar{d}_M-d_1-d_2>\s_m$ be a critical value. Then the
Hodge polynomial of the stable part $\cN_{\s_c}^s$ is
 $$
 \begin{aligned}
    e(\cN_{\s_c}^s)=
  \coeff_{x^0}
  \left[\frac{(1+u)^{2g}(1+v)^{2g}(1+xu)^{g}(1+xv)^{g}}{(1-uv)(1-x)(1-uv x)x^{d_1-d_2-\bar{d}_M}}
    \left(\frac{(uv)^{d_1-d_2-\bar{d}_M}}{1-(uv)^{-1}x} -
    \frac{(uv)^{-d_1+g+1+2\bar{d}_M}x}{1-(uv)^2x} -1 \right)\right] .
 \end{aligned}
 $$
\end{proposition}

\begin{proof}
From the Definition \ref{def:flip-loci}, we easily get that
$\cN_{\s_c}^s=\cN_{\scm}-\cS_{\scm}$, so
  $$
    e(\cN_{\s_c}^s)= e(\modulimenos)- e(\Smenos).
  $$
By part (2) in the proof of \cite[Lemma 6.1]{MOV1} (alternatively,
use Proposition \ref{prop:locus1}), $\cS_{\scm}$ is the
projectivization of a rank $-\chi(T'',T')$ bundle over
$\cN_{\s_c}'\x \cN_{\s_c}''$, where
\begin{align*}
    \cN'_{\s_c} &= \cN_{\s_c}(1,1,d_1-\bar{d}_M,d_2) =\Jac^{d_2}X \x
    \Sym^{d_1-\bar{d}_M-d_2}X\, , \\
    \cN''_{\s_c} &= \cN_{\s_c}(1,0,\bar{d}_M,0)= \Jac^{\bar{d}_M}X \, ,\\
    -\chi(T'',T') &=2\bar{d}_M-d_1+g-1\, .
\end{align*}
Therefore,
 $$
 \begin{aligned}
  e(\Smenos)
    &= e(\Jac \, X)^2 e(\Sym^{d_1-d_2-\bar{d}_M} X) e(\PP^{2\bar{d}_M-d_1+g-2}) \\
    &= (1+u)^{2g}(1+v)^{2g} \frac{1-(uv)^{2\bar{d}_M-d_1+g-1}}{1-uv}
    \coeff_{x^0} \frac{(1+ux)^{g}(1+vx)^g}
     {(1-x)(1-uvx) x^{d_1-d_2-\bar{d}_M}} \, .
 \end{aligned}
 $$

Also $e(\modulimenos)$ is computed by taking
$\s=\scm=\sigma_c-\epsilon$ ($\epsilon>0$ small) into the formula
of Theorem \ref{thm:polinomiono(2,1)no-critico}, with
$d_0=\Big[\frac13(\s+d_1+d_2)\Big]+1=\bar{d}_M$. Substracting both
terms, we get
 $$
 \begin{aligned}
    e(\cN_{\s_c}^s)=
  \coeff_{x^0}
  \Bigg[ & \frac{(1+u)^{2g}(1+v)^{2g}(1+xu)^{g}(1+xv)^{g}}{(1-uv)(1-x)(1-uv
  x)x^{d_1-d_2-\bar{d}_M}} \cdot
  \\
    & \cdot \Bigg(\frac{(uv)^{d_1-d_2-\bar{d}_M}}{1-(uv)^{-1}x}-
    \frac{(uv)^{-d_1+g-1+2\bar{d}_M}}{1-(uv)^2x}-(1-(uv)^{2\bar{d}_M-d_1+g-1}) \Bigg)\Bigg]
    ,
 \end{aligned}
 $$
and rearranging we get the stated result.
\end{proof}

Let $M(2,d)$ denote the moduli space of polystable vector bundles
of rank $2$ and degree $d$ over $X$. As $M(2,d)\cong M(2,d+2k)$,
for any integer $k$, there are two moduli spaces, depending on
whether de degree is even or odd. The Hodge polynomial of the
moduli space of rank $2$ odd degree stable bundles is given in
\cite{BaR,EK,MOV1}.

\begin{theorem}\label{thm:rank2odd}
The Hodge polynomial of $M(2,d)$ with odd degree $d$, is
 $$
    e(M(2,d)) = \frac{(1+u)^{g}(1+v)^g(1+u^2v)^{g}(1+uv^2)^g
    -(uv)^{g}(1+u)^{2g}(1+v)^{2g}}
    {(1-uv)(1-(uv)^2)}\,.
 $$ \hfill $\Box$
\end{theorem}

The Hodge polynomial of the moduli space of rank $2$ even degree
stable bundles is computed in \cite{MOV2}. Note that this moduli
space is smooth but non-compact.

\begin{theorem}[{\cite[Theorem A]{MOV2}}]\label{thm:rank2even}
The Hodge polynomial of $M^s(2,d)$ with even degree $d$, is
 \begin{align*}
   e(M^s(2,d))=\, &
  \frac{1}{2(1-uv)(1-(uv)^2)} \bigg( 2(1+u)^{g}(1+v)^g(1+u^2v)^{g}(1+uv^2)^g   \\
  &-
  (1+u)^{2g}(1+v)^{2g} (1+ 2 u^{g+1} v^{g+1} -u^2v^2) -
  (1-u^2)^g(1-v^2)^g(1-uv)^2  \bigg) \, .
 \end{align*}
\end{theorem}

%

\section{Critical values for triples of rank $(3,1)$}\label{sec:critical(3,1)}

Now we move to the analysis of the moduli spaces of
$\s$-polystable triples of rank $(3,1)$. Let $\cN_\s=
\cN_\s(3,1,d_1,d_2)$. By Proposition \ref{prop:alpha-range}, $\s$
takes values in the interval
 $$
 I=[\s_m,\s_M]=
 [\mu_1-\mu_2,3(\mu_1-\mu_2)]=[\frac{d_1}{3}-d_2, d_1-3d_2] \, .
 $$
Otherwise $\cN_\s$ is empty. Note that if $d_1-3d_2<0$, then $I$
is empty.

\begin{theorem} \label{thm:moduli}
For $\s\in I$, $\cN_\s$ is a projective variety of dimension $7g-6+d_1-3d_2$.
It is smooth at any $\s$-stable point. For non-critical values of $\s$,
we have $\cN_\s=\cN_\s^s$.
\end{theorem}

\begin{proof}
Projectiveness follows from Proposition \ref{prop:alpha-range}.
Smoothness and the dimension follow from Theorem
\ref{thm:smoothdim}.
\end{proof}

\begin{proposition} \label{prop:critical-values}
The critical values $\s_c$ for triples of type $(3,1,d_1,d_2)$
such that $\s_c>\s_m$ are the numbers
  $$
  \s_n=2 n -d_1-d_2, \ \ \frac23 d_1 < n \leq d_1-d_2\, , n\in \ZZ \, .
  $$
\end{proposition}

\begin{proof}
  Let $\sigma_c$ be a critical value and let $T=(E,L,\phi)$ be
  a properly $\sigma_c$-semistable triple of type $(3,1,d_1,d_2)$.
  Let $T'\subset T$ be a $\sigma_c$-destabilizing triple.
  Then we have the following cases:
  \begin{enumerate}
  \item If $T$ is $\sigma_c^+$-stable (i.e. $T\in \cN_\scp$) then
  it must be that $\lambda'<\lambda$. Therefore $T'$
  must be of type $(n_1',0)$, i.e., $T'=(G,0,0)$. This gives
  $$
    \frac{d_1+d_2}{4} +\frac14 \s_c =\mu(G) \, ,
  $$
  so $\s_c=4\mu(G) -d_1-d_2$. If $T'$ is of type $(3,0)$ then
  $G=E$ and $\s_c=4\frac{d_1}{3} -d_1-d_2 =\frac{d_1}3 -d_2=\s_m$.
  Otherwise, $T'$ is of type $(2,0)$ or $(1,0)$ :
   \begin{enumerate}
   \item If $T'=(M,0,0)$, where $M$ is a line bundle of degree $d_M$, we have
   $$
\begin{array}{ccc}
 0 & \flechad{25} & M\\
 \vert &  & \vert \\ [-5pt]
 \downarrow & & \downarrow \\
 L & \stackrel{\scriptstyle\phi}{\flechad{25}} & E_1 \\
 \vert &  & \vert \\ [-5pt]
 \downarrow & & \downarrow \\
 L & \flechad{25} & F
 \end{array}
   $$
   and $\s_c=4d_M-d_1-d_2$. This corresponds to $\s_c=\s_n$ for $n=2d_M$.
   \item If $T'=(F,0,0)$, where $F$ is a rank $2$ bundle of degree $d_F$, we have
   $$
\begin{array}{ccc}
0 & \flechad{25} & F\\
  \vert &  & \vert \\ [-5pt]
  \downarrow & & \downarrow \\
  L & \stackrel{\scriptstyle\phi}{\flechad{25}} & E_1 \\
  \vert &  & \vert \\ [-5pt]
 \downarrow & & \downarrow \\
  L & \flechad{25} & L'
 \end{array}
$$
 and $\s_c=2d_F-d_1-d_2$. This corresponds to $\s_c=\s_n$ for $n=d_F$.
   \end{enumerate}
\item If $T$ is $\sigma_c^-$-stable (i.e. $T\in \cN_\scm$) then
  it must be that $\lambda'>\lambda$. Therefore $T'$
  must be of type $(n_1',1)$. If $T'$ is of type $(0,1)$ then
  it should be $\phi=0$, which is not possible. So $T'$ is of type
  $(2,1)$ or $(1,1)$. The quotient triple is of the form
  $T''=(G,0,0)$, with $G$ of rank $1$ or $2$. So
  $$
    \frac{d_1+d_2}{4} +\frac14 \s_c =\mu(G)
  $$
  and $\s_c=4\mu(G) -d_1-d_2$.
   \begin{enumerate}
   \item If $T''=(M,0,0)$, where $M$ is a line bundle of degree $d_M$, we have
   $$
\begin{array}{ccc}
 L & \flechad{25} & F\\
 \vert &  & \vert \\ [-5pt]
 \downarrow & & \downarrow \\
 L & \stackrel{\scriptstyle\phi}{\flechad{25}} & E_1 \\
 \vert &  & \vert \\ [-5pt]
 \downarrow & & \downarrow \\
 0 & \flechad{25} & M
 \end{array}
   $$
   and $\s_c=4d_M-d_1-d_2$. This corresponds to $\s_c=\s_n$ for $n=2d_M$.
   \item If $T''=(F,0,0)$, where $F$ is a rank $2$ bundle of degree $d_F$, we have
   $$
\begin{array}{ccc}
L & \flechad{25} & M\\
  \vert &  & \vert \\ [-5pt]
  \downarrow & & \downarrow \\
  L & \stackrel{\scriptstyle\phi}{\flechad{25}} & E_1 \\
  \vert &  & \vert \\ [-5pt]
 \downarrow & & \downarrow \\
  0 & \flechad{25} & F
 \end{array}
$$
and $\s_c=2d_F-d_1-d_2$. This corresponds to $\s_c=\s_n$ for
$n=d_F$.
  \end{enumerate}
  \end{enumerate}
  Finally the condition $\s_m< \s_c \leq \s_M$, i.e.
  $\frac{d_1}{3}-d_2<
  2n -d_1-d_2 \leq d_1-3d_2$, is translated into $\frac23 d_1 < n \leq d_1-d_2$.
\end{proof}

Let us compute the contribution of a critical value
$\s_n=2n-d_1-d_2 >\s_m$ (that is $n > \frac23 d_1$) to the Hodge
polynomial
 \begin{equation}\label{eqn:Cn}
  C_n:= e(\cN_\scp)-e(\cN_\scm)=e(\cS_\scp)-e(\cS_\scm) \, .
 \end{equation}

Let us introduce the following notation
 \begin{eqnarray*}
  N_1 &=& d_1 - d_2 - n\, , \\
  N_2 &=& g-1  -d_1+ 3\frac{n}2\, .
 \end{eqnarray*}

\begin{proposition} \label{prop:Cnodd}
 If $n$ is odd then $C_n$ in (\ref{eqn:Cn}) equals
  $$
  \begin{aligned}
   C_n = & (1+u)^{2g}(1+v)^{2g}
   \coeff_{x^0}\frac{(1+ux)^{g}(1+vx)^{g}}{(1-x)(1-uvx)x^{d_1-d_2-n}}\cdot
   \\ & \cdot
    \big( (uv)^{2N_2}-(uv)^{2N_1} \big)
   \cdot  \frac{(1+u^2v)^{g}(1+uv^2)^g  -(uv)^{g}(1+u)^{g}(1+v)^{g}}
    {(1-uv)^2(1-(uv)^2)}\,.
   \end{aligned}
   $$
   \end{proposition}

\begin{proof}
If $n$ is odd, then the only possible cases for properly
$\s_c$-semistable triples are those in cases (1)(b) and (2)(b) of
Proposition \ref{prop:critical-values} with $d_F=n$ odd. Then $F$
is a semistable bundle of odd degree, hence stable. So $\cS_\scp$
consists of extensions as in (1)(b). Applying Proposition
\ref{prop:locus1} (which is possible since $\HH^2(T'',T')=0$ by
Lemma \ref{lem:H2=0}), $\cS_\scp$ is the projectivization of a
bundle of rank
  $$
  -\chi(T'',T') = 2d_1-2d_F -2d_2 = 2 N_1
  $$
over
 $$
 \cN_{\s_c}'^{,s}\x \cN_{\s_c}''^{,s} = \cN_{\s_c}^s(2,0,d_F,0) \x
 \cN_{\s_c}^s(1,1,d_1-d_F,d_2)\,.
 $$
Hence
   \begin{equation}\label{eqn:a1}
   e(\cS_\scp) = e(\PP^{2N_1-1})\,
   e(\cN_{\s_c}^s(1,1,d_1-d_F,d_2))\,
   e(M(2,d_F))\, .
   \end{equation}

Analogously,  $\cS_\scm$ consists of extensions as in (2)(b).
Hence, by Proposition \ref{prop:locus1}, $\cS_\scm$ is the
projectivization of a bundle of rank
   $$
   -\chi(T'',T') = 2g-2+3d_F-2d_1 =2 N_2
   $$
over
 $$
  \cN_{\s_c}'^{,s}\x \cN_{\s_c}''^{,s} =  \cN_{\s_c}^s(1,1,d_1-d_F,d_2) \x
  \cN_{\s_c}^s(2,0,d_F,0)\, .
 $$
   Therefore
   \begin{equation}\label{eqn:a2}
   e(\cS_\scm) = e(\PP^{2N_2-1})\, e(\cN_{\s_c}^s(1,1,d_1-d_F,d_2)) \,
   e(M(2,d_F))\, .
   \end{equation}

Now recall that $d_F=n$ and note that
$\cN_{\s_c}^s(1,1,d_1-d_F,d_2) \cong \Jac^{d_2} X \x
\Sym^{d_1-d_2-n} X$ (by \cite[Lemma 3.16]{MOV1}, since $\s_c\neq
\s_m(1,1,d_1-d_F,d_2)$). Substracting (\ref{eqn:a1}) from
(\ref{eqn:a2}) we get
   $$
   \begin{aligned}
   C_n =& \, (e(\PP^{2N_1-1}) - e(\PP^{2N_2-1}))\, e(\cN_{\s_c}^s(1,1,d_1-d_F,d_2))\, e(M(2,d_F)) \\
   = & \, \frac{(uv)^{2N_2}-(uv)^{2N_1}}{1-uv} (1+u)^g(1+v)^g
 \coeff_{x^0}\frac{(1+ux)^{g}(1+vx)^{g}}{(1-x)(1-uvx)x^{d_1-n-d_2}} \\
    & \cdot  \frac{(1+u)^{g}(1+v)^g(1+u^2v)^{g}(1+uv^2)^g
    -(uv)^{g}(1+u)^{2g}(1+v)^{2g}}
    {(1-uv)(1-(uv)^2)}\,.
   \end{aligned}
   $$
  \end{proof}

\begin{proposition} \label{prop:Cneven}
 If $n$ is even  then $C_n$ in (\ref{eqn:Cn}) equals
 $$
     \begin{aligned}
  C_n =&
 (1+u)^{2g}(1+v)^{2g}
  \coeff_{x^0}\frac{(1+ux)^{g}(1+vx)^{g}}{(1-x)(1-uvx)x^{d_1-d_2-n}}
  \cdot
   \\ &
  \cdot \Bigg[\big( (uv)^{2N_2}-(uv)^{2N_1} \big)
  \frac{(1+u^2v)^{g}(1+uv^2)^g-(uv)^{g}(1+u)^{g}(1+v)^g}{(1-uv)^2(1-(uv)^2)}
  \\
  & \quad - \frac{(uv)^{g-1}(1+u)^{g}(1+v)^{g}}{(1-uv)^2(1+uv)}
    \Bigg(\frac{(uv)^{2N_1+1} (1+(uv)^{-2}x) }{1-(uv)^{-1}x} +
    \frac{(uv)^{2N_2} (1+(uv)^3x)}{1-(uv)^2x} \\
   & \qquad \qquad \qquad -\frac{(uv)^{N_1+N_2}(1+uv)(1-(uv)x^{2})}{(1-(uv)^{-1}x)(1-(uv)^2x)} \Bigg)
 \Bigg]\,.
 \end{aligned}
 $$
\end{proposition}

\begin{proof}
  Let $\s_c=\s_n=2n-d_1-d_2$. Any triple $T\in \cS_\scp$ (resp.\  $T\in
  \cS_\scm$) sits in a non-split exact sequence $T'\to T\to T''$ such that
  $T'$ is $\s_c$-semistable, $T''$ is $\s_c$-stable and
  $\lambda'<\lambda$ (resp.\ $T'$ is $\s_c$-stable, $T''$ is $\s_c$-semistable and
  $\lambda''<\lambda$). Since $\lambda=\frac14$, it must be
  $\lambda'=0$ (resp.\ $\lambda''=0$).
  Therefore,
  we can decompose $\cS_\scp$ (resp.\ $\cS_\scm$) into $6$ disjoint algebraic locally closed
  subspaces
  $\cS_\scp=X_1^+\cup X_2^+\cup X_3^+\cup X_4^+\cup X_5^+\cup X_6^+$
  (resp.\ $\cS_\scm=X_1^-\cup X_2^-\cup X_3^-\cup X_4^-\cup X_5^-\cup X_6^-$),
  as follows
  \begin{itemize}
  \item $X_1^+$ (resp.\ $X_1^-$) consists of those extensions of type (1)(a) (resp.\ (2)(a))
  for which $T''=(F,L,\phi'')$ is a $\s_c$-stable triple (resp.\ $T'=(F,L,\phi')$ is a $\s_c$-stable triple).
  \item $X_2^+$ (resp.\ $X_2^-$) consists of those extensions of type (1)(b) (resp.\ (2)(b)) for which $F$ is a stable bundle
  of degree $d_F=n$.
  \item $X_3^+$ (resp.\ $X_3^-$) consists of those extensions of type (1)(b) (resp.\ (2)(b)) for which $F$ is a properly
  semistable bundle of degree $d_F=n$, sitting in a non-split exact sequence
  $L_1\to F\to L_2$, where $L_1\not\cong L_2$, $L_1,L_2\in \Jac^{n/2} X$.
  \item $X_4^+$ (resp.\ $X_4^-$) consists of those extensions of type (1)(b) (resp.\ (2)(b)) for which $F$ is a properly
  semistable bundle of degree $d_F=n$, sitting in a non-split exact sequence
  $L_1\to F\to L_1$, where $L_1\in \Jac^{n/2} X$.
  \item $X_5^+$ (resp.\ $X_5^-$) consists of those extensions of type (1)(b) (resp.\ (2)(b)) for which $F$ is a properly
  semistable bundle of degree $d_F=n$ of the form
  $F= L_1\oplus L_2$, where $L_1\not\cong L_2$, $L_1,L_2\in \Jac^{n/2} X$.
  \item $X_6^+$ (resp.\ $X_6^-$) consists of those extensions of type (1)(b) (resp.\ (2)(b)) for which $F$ is a properly
  semistable bundle of degree $d_F=n$ of the form
  $F= L_1\oplus L_1$, where $L_1\in \Jac^{n/2} X$.
  \end{itemize}

  We aim to compute
   $$
   \begin{aligned}
      C_n &= e(\cS_\scp)-e(\cS_\scm) \\
   & =e(\bigsqcup X_i^+)-e(\bigsqcup X_i^-) \\
   & = \sum e( X_i^+)- \sum e(X_i^-) \\
   & = \sum (e( X_i^+)- e(X_i^-) )\, .
   \end{aligned}
      $$
We shall do this by computing each of the terms in the sum above
independently:
   \begin{enumerate}
   \item We compute $e( X_1^+)- e(X_1^-)$ as follows. Proposition
   \ref{prop:locus1} implies that $X_1^+$ is the projectivization
   of a bundle of rank (with $T'$ and $T''$ as in (1)(a))
   $$
   -\chi(T'',T') = g-1 + d_1- 2 d_M -d_2 = g-1 +d_1-d_2-n = g-1 + N_1\,.
   $$
   over
   $$
   \cN_{\s_c}'^{,s} \x \cN_{\s_c}''^{,s} =
    \cN^s_{\s_c}(1,0,d_M,0) \x \cN^s_{\s_c}(2,1,d_1-d_M,d_2)\, .
   $$
   Note that here $d_M=n/2$, and the hypothesis of Proposition
   \ref{prop:locus1} is satisfied because of Lemma \ref{lem:H2=0}.
   Therefore
   \begin{equation}\label{eqn:b1}
   e(X_1^+) = e(\PP^{g-1+N_1-1})\,  e(\cN_{\s_c}^s(2,1,d_1-d_M,d_2)) \, e(\Jac^{d_M}
   X)\,.
   \end{equation}

   Analogously, $X_1^-$ is the projectivization
   of a bundle of rank (now $T'$ and $T''$ as in (2)(a))
   $$
    -\chi(T'',T') = 2g-2+3d_M- d_1 = 2g-2 -d_1 +3n/2 =g-1+N_2
   $$
   over
   $$
   \cN_{\s_c}'^{,s} \x \cN_{\s_c}''^{,s}
   =\cN^s_{\s_c}(2,1,d_1-d_M,d_2)\x
    \cN^s_{\s_c}(1,0,d_M,0) \, .
   $$
   Therefore
   \begin{equation}\label{eqn:b2}
   e(X_1^-) = e(\PP^{g-1+N_2-1})\, e(\cN_{\s_c}^s(2,1,d_1-d_M,d_2)) \, e(\Jac^{d_M}
   X)\, .
   \end{equation}

   Proposition \ref{prop:polinomio(2,1)critico} says that (using $\bar{d}_M=\frac13 ( \s_c+d_1-d_M
   +d_2) =\frac13 ( 2n-d_1-d_2 +d_1 - \frac{n}2 +d_2) = \frac{n}2$
   and $d_M=\frac{n}2$),
   \begin{equation}\label{eqn:b3}
   \begin{aligned}
    e(\cN_{\s_c}^s(2,1,d_1-d_M,d_2)) =
    \coeff_{x^0}
    \Bigg[ &\frac{(1+u)^{2g}(1+v)^{2g}(1+xu)^{g}(1+xv)^{g}}{(1-uv)(1-x)(1-uv
    x)x^{d_1-d_2-n}}\cdot
    \\ & \cdot \Bigg(\frac{(uv)^{d_1-d_2-n}}{1-(uv)^{-1}x}   -
    \frac{(uv)^{-d_1+g+1+3n/2}x}{1-(uv)^2x}-1\Bigg)\Bigg].
   \end{aligned}
   \end{equation}
   Note that $\s_c >\s_m(2,1,d_1-d_M,d_2)=(d_1-d_M)/2-d_2$ (which follows from
   $n>\frac23 d_1$), so Proposition \ref{prop:polinomio(2,1)critico} applies.

   Substracting (\ref{eqn:b2}) from (\ref{eqn:b1}) and using (\ref{eqn:b3}), we get
   $$
    \begin{aligned}
    e(X_1^+)-e(X_1^-) = & (e(\PP^{g-1+N_1-1}) - e(\PP^{g-1+N_2-1})) e(\cN_{\s_c}^s(2,1,d_1-d_M,d_2)) e(\Jac^{d_M} X) \\
      = &\frac{(uv)^{g-1} ((uv)^{N_2}-(uv)^{N_1}) }{1-uv} (1+u)^g(1+v)^g \\
    & \cdot \,
    \coeff_{x^0}
    \left[\frac{(1+u)^{2g}(1+v)^{2g}(1+xu)^{g}(1+xv)^{g}}{(1-uv)(1-x)(1-uv x)x^{d_1-d_2-n}}
    \Bigg(\frac{(uv)^{N_1}}{1-(uv)^{-1}x} -
    \frac{(uv)^{N_2+2}x}{1-(uv)^2x}-1\Bigg)\right].
    \end{aligned}
   $$

   \item We compute now $e( X_2^+)- e(X_2^-)$.
   Any $T\in X_2^+$ is a non-split extension $T'\to T \to T''$, where both
   $T'$ and $T''$ are $\s_c$-stable. Moreover, $\HH^2(T'',T')=0$ by Lemma \ref{lem:H2=0}. So
   we use Proposition \ref{prop:locus1} to get that $X_2^+$ is the
   projectivization of a fiber bundle of rank (with $T'$ and $T''$ as in (1)(b))
   $$
   -\chi(T'',T') = 2d_1-2d_F -2d_2 =2N_1
   $$
   over
   $$
   \cN_{\s_c}'^{,s} \x \cN_{\s_c}''^{,s} = M^s(2, d_F) \x \cN_{\s_c}^s (1,1,d_1-d_F,d_2)\, .
   $$
   Note that $d_F=n$. Therefore,
   \begin{equation}\label{eqn:c1}
   e(X_2^+) = e(\PP^{2N_1-1}) \, e(\cN_{\s_c}^s(1,1,d_1-d_F,d_2)) \,
   e(M^s(2,d_F))\, .
   \end{equation}

   Analogously,  $X_2^-$ is the
   projectivization of a fiber bundle of rank (now $T'$ and $T''$ as in (2)(b))
   $$
   -\chi(T'',T') = 2g-2+3d_F-2d_1 =2N_2
   $$
   over
   $$
   \cN_{\s_c}'^{,s} \x \cN_{\s_c}''^{,s} =\cN_{\s_c}^s (1,1,d_1-d_F,d_2)\x M^s(2, d_F) \, .
   $$
   Therefore,
   \begin{equation}\label{eqn:c2}
   e(X_2^-) = e(\PP^{2N_2-1}) \, e(\cN_{\s_c}^s(1,1,d_1-d_F,d_2)) \,
   e(M^s(2,d_F))\, .
   \end{equation}

   Now note that $\cN_{\s_c}^s(1,1,d_1-d_F,d_2) \cong \Jac^{d_2} X \x \Sym^{d_1-d_F-d_2}
   X$ (using \cite[Lemma 3.16]{MOV1}, since $\s_c\neq
   \s_m(1,1,d_1-d_F,d_2)$). Substracting (\ref{eqn:c2}) from (\ref{eqn:c1}) we get
   $$
    \begin{aligned}
     e(X_2^+)-e(X_2^-) =& (e(\PP^{2N_1-1}) - e(\PP^{2N_2-1}))e(\cN_{\s_c}^s(1,1,d_1-d_F,d_2)) e(M^s(2,d_F)) \\
       =& \frac{(uv)^{2N_2}-(uv)^{2N_1}}{1-uv}\,  (1+u)^g(1+v)^g \cdot \,
    \coeff_{x^0}\frac{(1+ux)^{g}(1+vx)^{g}}{(1-x)(1-uvx)x^{d_1-d_2-n}} \cdot \\
    & \cdot
    \frac{1}{2(1-uv)(1-(uv)^2)} \Big( 2(1+u)^{g}(1+v)^g(1+u^2v)^{g}(1+uv^2)^g  \\ &-
  (1+u)^{2g}(1+v)^{2g} (1+ 2 u^{g+1} v^{g+1} -u^2v^2) -
  (1-u^2)^g(1-v^2)^g(1-uv)^2 \Big)\,.
    \end{aligned}
   $$

   \item Now we compute $e(X_3^+)- e(X_3^-)$. An element $T \in
   X_3^+$ sits in a non-split extension $T'\to T\to T''$, where $T_1\to T'\to
   T_2$ is non-split, and all $T'', T_1,T_2$ are $\s_c$-stable of the same
   $\s_c$-slope, and $T_1\not\cong T_2$. Note that
   $T_i=(L_i,0,0)$, $i=1,2$, since we are in the situation of
   (1)(b). So $\lambda_1=0$, $\lambda_2=0$.
   By Lemma \ref{lem:H2=0}, we have that
   $\HH^2(T_2,T_1)=\HH^2(T'',T_2)=\HH^2(T'',T_1)=0$.
   We apply Proposition \ref{prop:locus2}, obtaining that $X_3^+$ is a bundle over
   $Y\x \cN_{\s_c}''^{,s}$, where $Y$ parametrizes the triples
   $T'$, and with fiber $\PP^{a-1}-\PP^{b-1}$, where
    $$
    \begin{aligned}
    a &= -\chi(T'',T')  = 2d_1-4d_M -2d_2 = 2N_1 \, ,\\
    b &= -\chi(T'',T_1) = d_1-2d_M -d_2 =N_1\, .
    \end{aligned}
    $$
   The space $Y$ parametrizing non-split extensions $T_1\to T'\to
   T_2$ is a projective bundle with fiber $\PP^{g-2}$, since
    $$
    -\chi(T_2,T_1)=g-1\, ,
    $$
   over
    $$
    \qquad \{ (T_1,T_2) \in \cN_{\s_c}(1,0,{n/2},0) \x\cN_{\s_c}(1,0,{n/2},0) \ |
    \ T_1\not\cong T_2 \} = (\Jac^{n/2} X \x \Jac^{n/2} X )-\Delta \, ,
    $$
   where $\Delta$ stands for the diagonal. Therefore
       \begin{equation}\label{eqn:d1}
     e(X_3^+) = (e(\PP^{2N_1-1}) - e(\PP^{N_1-1})) e(\Jac^{d_2} X) e(\Sym^{d_1-2d_M-d_2})
      e( (\Jac X\x \Jac X)-\Delta) e(\PP^{g-2}) \, .
       \end{equation}

  The case of $X_3^-$ is analogous. An element $T \in
   X_3^-$ sits in a non-split extension $T'\to T\to T''$, where $T_1\to T''\to
   T_2$ is non-split, and all $T', T_1,T_2$ are $\s_c$-stable of the same
   $\s_c$-slope, and $T_1\not\cong T_2$. Again
   $T_i=(L_i,0,0)$, $i=1,2$.
   Proposition \ref{prop:locus2} yields that $X_3^-$ is a bundle over
   $$
   \cN_{\s_c}'^{,s} \x Y =\Jac^{d_2} X \x \Sym^{d_1-2d_M-d_2} \x
   Y\, ,
   $$
   where $Y$ parametrizes the triples $T''$ (therefore $Y$ is a
   $\PP^{g-2}$-bundle over $(\Jac^{n/2} X \x \Jac^{n/2} X
   )-\Delta$), and with fiber $\PP^{a-1}-\PP^{b-1}$, where
    $$
    \begin{aligned}
    a &= -\chi(T'',T')= 2g-2 +6 d_M - 2 d_1 = 2N_2 \, ,\\
    b &= -\chi(T_2,T') = g-1 +3 d_M -d_1 =N_2 \, .
    \end{aligned}
    $$
   Therefore
     \begin{equation}\label{eqn:d2}
     e(X_3^-) = (e(\PP^{2N_2-1}) - e(\PP^{N_2-1})) e(\Jac^{n/2} X) e(\Sym^{d_1-2d_M-d_2})
      e( (\Jac X\x \Jac X)-\Delta) e(\PP^{g-2}) \, .
     \end{equation}

    Substracting (\ref{eqn:d2}) from (\ref{eqn:d1}), we get
    $$
    \begin{aligned}
    e( X_3^+)- e(X_3^-) = & (e(\PP^{2N_1-1}) - e(\PP^{N_1-1}) - e(\PP^{2N_2-1}) +
    e(\PP^{N_2-1})) \cdot \\
    & \cdot  e(\Jac^{d_2} X) e(\Sym^{d_1-2d_M-d_2})
      e(   (\Jac X \x \Jac X)-\Delta) e(\PP^{g-2}) \\
    = & \frac{(uv)^{2N_2}-(uv)^{N_2}-(uv)^{2N_1}+(uv)^{N_1}}{1-uv}
    \, (1+u)^g(1+v)^g \cdot
    \\
    & \cdot \coeff_{x^0} \frac{(1+ux)^{g}(1+vx)^{g}}{(1-x)(1-uvx)x^{d_1-d_2-n}}
    \cdot  ((1+u)^{2g}(1+v)^{2g} -(1+u)^g(1+v)^g) \,
    \frac{1-(uv)^{g-1}}{1-uv}\,.
    \end{aligned}
    $$

    \item We move on to compute $e( X_4^+)- e(X_4^-)$. An element $T\in
    X_4^+$ sits in a non-split extension $T'\to T\to T''$, where
    $T_1\to T'\to T_1$ is non-split, and $T'', T_1$ are
    $\s_c$-stable triples of the same $\s_c$-slope. The triples
    $T_1=(L_1,0,0)$ are parametrized by $\Jac^{n}X$. The triples
    $T'$ are then parametrized by a variety $Y$ which is a
    projective bundle over $\Jac^{n/2} X$ with fiber projective spaces
    $\PP^{g-1}$, since $\chi(T_1,T_1)=g$.
    By Lemma \ref{lem:H2=0}, we have that $\HH^2(T_1,T_1)=\HH^2(T'',T_1)=0$.
    By Proposition \ref{prop:locus3}, $X_4^+$ is a bundle over
    $$
    \cN_{\s_c}''^{,s} \x Y =\Jac^{d_2} X \x \Sym^{d_1-2d_M-d_2} X \x
    Y\, ,
    $$
    with fiber a $\CC^{N_1-1}$-bundle over $\PP^{N_1-1}$, as
    $$
    -\chi(T'',T_1) = d_1-2d_M -d_2 =N_1 \, .
    $$
    Therefore
    \begin{equation}\label{eqn:e1}
     e(X_4^+) = e(\CC^{N_1-1}) e(\PP^{N_1-1}) e(\Jac^{d_2} X) e(\Sym^{d_1-2d_M-d_2})
     e( \Jac^{n/2} X) e(\PP^{g-1}) \, .
     \end{equation}

    Analogously, an element $T\in
    X_4^-$ sits in a non-split extension $T'\to T\to T''$, where
    $T_1\to T''\to T_1$ is non-split, and $T', T_1$ are
    $\s_c$-stable triples of the same $\s_c$-slope. The triples
    $T''$ are parametrized by the variety $Y$ as above.
    Proposition \ref{prop:locus3} implies that $X_4^-$ is a bundle over
    $$
    Y \x \cN_{\s_c}'^{,s} = Y \x \Jac^{d_2} X \x
    \Sym^{d_1-2d_M-d_2}X\, ,
    $$
    with fiber a $\CC^{N_2-1}$-bundle over $\PP^{N_2-1}$, as
    $$
    -\chi(T_1,T') = g-1 +3 d_M -d_1 =N_2\, .
    $$
    Hence
     \begin{equation}\label{eqn:e2}
     e(X_4^-) = e(\CC^{N_2-1}) e(\PP^{N_2-1}) e(\Jac^{d_2} X) e(\Sym^{d_1-2d_M-d_2})
      e(  \Jac^{n/2} X) e(\PP^{g-1}) \, .
     \end{equation}

Combining (\ref{eqn:e1}) with (\ref{eqn:e2}) we get
    $$
    \begin{aligned}
    e( X_4^+)- e(X_4^-) = &  (e(\CC^{N_1-1}) e(\PP^{N_1-1}) -e(\CC^{N_2-1}) e(\PP^{N_2-1}))
    e(\Jac^{d_2} X) \cdot \\ &\cdot e(\Sym^{d_1-2d_M-d_2})
      e(  \Jac^{n/2} X) e(\PP^{g-1}) \\
       = & \frac{(uv)^{2N_2-1}-(uv)^{N_2-1}-(uv)^{2N_1-1}+(uv)^{N_1-1}}{1-uv}
    (1+u)^g(1+v)^g
     \\
    & \cdot \coeff_{x^0} \frac{(1+ux)^{g}(1+vx)^{g}}{(1-x)(1-uvx)x^{d_1-d_2-n}}\cdot  (1+u)^g(1+v)^g
    \frac{1-(uv)^{g}}{1-uv}\,.
    \end{aligned}
    $$

    \item We proceed with $e( X_5^+)- e(X_5^-)$. An element $T\in X^+_5$ sits in a non-split
    exact sequence $T'\to T\to T''$ where $T'=T_1\oplus T_2$, with $T_1,T_2$ non-isomorphic triples,
    and $T',T_1,T_2$ $\s_c$-stable triples of the same
    $\s_c$-slope. Here $T_i=(L_i,0,0)$ as in (1)(b). So the space
    parametrizing $T'$ is
    $$
    Y= (\Jac^{n/2} X\x \Jac^{n/2} X -\Delta)/\ZZ_2\, ,
    $$
    where $\ZZ_2$ acts by permutation $(T_1,T_2)\mapsto
    (T_2,T_1)$.

    By Lemma \ref{lem:H2=0},
    $\HH^2(T'',T_1)=\HH^2(T'',T_2)=0$. Then Proposition
    \ref{prop:locus4} implies that $X_5^+$ is a bundle over
    $S \x Y$, $S=\cN_{\s_c}''^{,s}$, whose fibers are $\PP^{a-1}\x
    \PP^{a-1}$, with
    $$
    \begin{aligned}
    a\ &=-\chi(T'',T_1)= d_1-2d_M -d_2 =N_1 \, .
    \end{aligned}
    $$
    However, this fiber bundle is not locally trivial in the
    Zariski topology, since it has monodromy around the diagonal.
    We compute the Hodge polynomial of $X_5^+$ as follows.
    Pull-back the bundle to $\widetilde{Y}=S\x (\Jac^{n/2} X\x \Jac^{n/2} X
    -\Delta)$,
    $$
 \begin{array}{ccc}
 \widetilde{X}_5^+ & \flechad{25} & X_5^+\\
 \vert &  & \vert \\ [-5pt]
 \downarrow & & \downarrow \\
 \widetilde{Y} & \flechad{25} & Y
 \end{array}
   $$
   Then $\widetilde{X}_5^+$ is a $\PP^{N_1-1}\x\PP^{N_1-1}$-bundle over $\widetilde{Y}$ and
   $X_5^+=\widetilde{X}_5^+/\ZZ_2$. More explicitly, let $Z$ be
   the projective bundle over $S\x \Jac^{n/2} X$, with fibers
   $\Ext^1(T'',T_1)$, and let $p: Z\x_S Z\to S\x \Jac^{n/2} X\x \Jac^{n/2} X$
   stand for the projection.
   Then $\widetilde{X}_5^+=p^{-1}( S\x (\Jac^{n/2} X\x \Jac^{n/2} X
    -\Delta))$. Letting $\ZZ_2$ act on $Z\x_S Z$ by permutation,
    we have an induced map $p: (Z\x_S Z)/\ZZ_2 \to (S\x \Jac^{n/2} X\x \Jac^{n/2} X)/\ZZ_2$
    and $X_5^+=p^{-1}(Y)$. Now
   \begin{equation}\label{eqn:first}
     e((Z\x_S Z)/\ZZ_2) = e(S)\cdot \frac12 \left( e(\PP^{N_1-1})^2 e(\Jac X)^2 +
     \frac{1-(uv)^{2N_1}}{1-(uv)^2} (1-u^2)^g(1-v^2)^g\right)\, ,
   \end{equation}
   using (7) of Section \ref{sec:virtual}.
   Also $p^{-1}(S\x\Delta)$ is a bundle over $S\x\Delta \cong S\x\Jac^{n/2} X$,
   whose fibers are $(\PP^{N_1-1}\x \PP^{N_1-1})/\ZZ_2$, where
   $\ZZ_2$ acts by permutation. Hence
   \begin{equation}\label{eqn:second}
    e(p^{-1}(S\x\Delta))=
    e(S)\,  e(\Jac X) \cdot \frac12 \left( e(\PP^{N_1-1})^2 +
     \frac{1-(uv)^{2N_1}}{1-(uv)^2} \right) \, ,
   \end{equation}
   using (7) of Section \ref{sec:virtual} again. We finally get
   \begin{equation}\label{eqn:f1}
     e(X_5^+)=  e((Z\x_S Z)/\ZZ_2)-e(p^{-1}(S\x\Delta)) = (\ref{eqn:first})-(\ref{eqn:second}) \, ,
   \end{equation}
   and $e(S)=e(\Jac^{d_2} X) e(\Sym^{d_1-d_F-d_2} X)$.

   There is an analogous formula for $e(X_5^-)$, obtained by
   substituting $N_1$ by $N_2$ in (\ref{eqn:f1}). So we obtain
    $$
    \begin{aligned}
    e( X_5^+)- e(X_5^-) =&  (1+u)^g(1+v)^g \,
    \coeff_{x^0} \frac{(1+ux)^{g}(1+vx)^{g}}{(1-x)(1-uvx)x^{d_1-d_2-n}} \cdot \\
    & \cdot \frac12 \left( \frac{(1-(uv)^{N_1})^2}{(1-uv)^2} (1+u)^{2g}(1+v)^{2g} +
     \frac{1-(uv)^{2N_1}}{1-(uv)^2} (1-u^2)^g(1-v^2)^g \right.\\
     & \quad   -(1+u)^g(1+v)^g  \frac{(1-(uv)^{N_1})^2}{(1-uv)^2}
     - (1+u)^g(1+v)^g\frac{1-(uv)^{2N_1}}{1-(uv)^2}  \\
     & \quad - \frac{(1-(uv)^{N_2})^2}{(1-uv)^2} (1+u)^{2g}(1+v)^{2g} -
     \frac{1-(uv)^{2N_2}}{1-(uv)^2} (1-u^2)^g(1-v^2)^g \\ &
       \quad \left. +(1+u)^g(1+v)^g  \frac{(1-(uv)^{N_2})^2}{(1-uv)^2} +
     (1+u)^g(1+v)^g\frac{1-(uv)^{2N_2}}{1-(uv)^2} \right) \\
       = &  (1+u)^g(1+v)^g \,
    \coeff_{x^0} \frac{(1+ux)^{g}(1+vx)^{g}}{(1-x)(1-uvx)x^{d_1-d_2-n}} \cdot \\
    & \cdot \frac12 \left( \frac{(uv)^{2N_1}-2(uv)^{N_1}-
    (uv)^{2N_2}+2(uv)^{N_2}}{(1-uv)^2} (
    (1+u)^{2g}(1+v)^{2g}- (1+u)^{g}(1+v)^{g}) \right.
    \\
    & \quad \left. + \frac{(uv)^{2N_2}-(uv)^{2N_1}}{1-(uv)^2} ( (1-u^2)^g(1-v^2)^g
     - (1+u)^g(1+v)^g) \right) \,.
    \end{aligned}
    $$

    \item We deal with the last case, $e( X_6^+)- e(X_6^-)$. A
    triple $T\in X_6^+$ sits in a non-split extension $T'\to T\to
    T''$, where $T'=T_1\oplus T_1$ and $T'',T_1$ are $\s_c$-stable
    triples with the same $\s_c$-slope. Here $T_1=(L_1,0,0)$ as in (1)(b).
    Using that $\HH^1(T'',T_1)=0$
    (by Lemma \ref{lem:H2=0}), Proposition \ref{prop:locus5} implies
    that $X_6^+$ is a grassmannian bundle over
     $$
     \cN_{\s_c}'' \x \Jac^{n/2} X =\Jac^{d_2} X \x \Sym^{d_1-d_F-d_2} X \x \Jac^{n}
     X\, ,
     $$
    with fibers $\Gr(2,N_1)$,
    since
    $$
    -\chi(T'',T_1)= N_1 \, .
    $$
    Hence
    $$
     e(X_6^+)=  e(\Jac^{d_2} X) e(\Sym^{d_1-d_F-d_2} X) e(\Jac^{d_M} X) e(\Gr(2,
     N_1))\, .
     $$

     Analogously,
     $$
     e(X_6^-)=  e(\Jac^{d_2} X) e(\Sym^{d_1-d_F-d_2} X) e(\Jac^{d_M} X) e(\Gr(2,
     N_2))\, ,
     $$
     and then
     $$
    \begin{aligned}
    e( X_6^+)- e(X_6^-) =& e(\Jac^{d_2} X) e(\Sym^{d_1-d_F-d_2} X) e(\Jac^{d_M} X)
    \big(e(\Gr(2, N_1))-e(\Gr(2, N_2))\big)\\
       =& (1+u)^{2g}(1+v)^{2g}
    \coeff_{x^0} \frac{(1+ux)^{g}(1+vx)^{g}}{(1-x)(1-uvx)x^{d_1-d_2-n}} \cdot \\
    & \cdot \frac{(1-(uv)^{N_1})(1-(uv)^{N_1-1}) -(1-(uv)^{N_2})(1-(uv)^{N_2-1})}{(1-uv)(1-(uv)^2)}
    \\
    =& (1+u)^{2g}(1+v)^{2g}
    \coeff_{x^0} \frac{(1+ux)^{g}(1+vx)^{g}}{(1-x)(1-uvx)x^{d_1-d_F-d_2}} \cdot \\
    & \cdot \frac{(uv)^{N_2}+(uv)^{N_2-1}-(uv)^{2N_2-1} -(uv)^{N_1}-(uv)^{N_1-1}+(uv)^{2N_1-1}}{(1-uv)(1-(uv)^2)} \,.
    \end{aligned}
    $$
\end{enumerate}

  Putting all together,
  $$
     \begin{aligned}
  C_n =&  (1+u)^g(1+v)^g
  \coeff_{x^0}\frac{(1+ux)^{g}(1+vx)^{g}}{(1-x)(1-uvx)x^{d_1-n-d_2}}
   \, \cdot
  \\ &
  \cdot \Bigg[  \frac{(uv)^{N_2}-(uv)^{N_1}}{1-uv} (uv)^{g-1}
   \frac{(1+u)^{2g}(1+v)^{2g}}{(1-uv)}
    \Bigg(\frac{(uv)^{N_1}}{1-(uv)^{-1}x} -
    \frac{(uv)^{N_2+2}x}{1-(uv)^2x}-1\Bigg) \\
    &  + \frac{(uv)^{2N_2}-(uv)^{2N_1}}{1-uv}  \cdot \frac{1} {2(1-uv)(1-(uv)^2)}
    \cdot \Big( 2(1+u)^{g}(1+v)^g(1+u^2v)^{g}(1+uv^2)^g  \\
    &\qquad -
  (1+u)^{2g}(1+v)^{2g} (1+ 2 (uv)^{g+1} -(uv)^2) -
  (1-u^2)^g(1-v^2)^g(1-uv)^2  \Big)
   \\
  &   + \frac{(uv)^{2N_2}-(uv)^{N_2}-(uv)^{2N_1}+(uv)^{N_1}}{1-uv}
  ((1+u)^{2g}(1+v)^{2g} -(1+u)^g(1+v)^g)
    \frac{1-(uv)^{g-1}}{1-uv}\\
   &  + \frac{(uv)^{2N_2-1}-(uv)^{N_2-1}-(uv)^{2N_1-1}+(uv)^{N_1-1}}{1-uv}
    (1+u)^g(1+v)^g \frac{1-(uv)^{g}}{1-uv}\\
  &   + \frac12 \Bigg( \frac{(uv)^{2N_1}-2(uv)^{N_1}-
    (uv)^{2N_2}+2(uv)^{N_2}}{(1-uv)^2} (
    (1+u)^{2g}(1+v)^{2g}- (1+u)^{g}(1+v)^{g})
    \\
    & \qquad + \frac{(uv)^{2N_2}-(uv)^{2N_1}}{1-(uv)^2} ( (1-u^2)^g(1-v^2)^g
     - (1+u)^g(1+v)^g) \Bigg) \\
   &   + (1+u)^{g}(1+v)^{g}
   \frac{(uv)^{N_2}+(uv)^{N_2-1}-(uv)^{2N_2-1} -(uv)^{N_1}-(uv)^{N_1-1}+(uv)^{2N_1-1}}{(1-uv)(1-(uv)^2)}   \Bigg]
   \,,
 \end{aligned}
 $$
 Rearranging, we get
  $$
     \begin{aligned}
  C_n =& (1+u)^{g}(1+v)^{g}
  \coeff_{x^0}\frac{(1+ux)^{g}(1+vx)^{g}}{(1-x)(1-uvx)x^{d_1-d_2-n}}
  \cdot
   \\ &
  \cdot \Bigg[  \frac{(uv)^{N_2}-(uv)^{N_1}}{1-uv} (uv)^{g-1}
   \frac{(1+u)^{2g}(1+v)^{2g}}{1-uv}
    \Bigg(\frac{(uv)^{N_1}}{1-(uv)^{-1}x} -
    \frac{(uv)^{N_2+2}x}{1-(uv)^2x}-1\Bigg) \\
 &+( (uv)^{2N_2}-(uv)^{2N_1} )\frac{(1+u)^{g}(1+v)^{g}(1+u^2v)^{g}(1+uv^2)^g
  -(uv)^{g-1}(1+u)^{2g}(1+v)^{2g}}{(1-uv)^2(1-(uv)^2)} \\
 &+( (uv)^{N_2}-(uv)^{N_1} )(1+u)^{2g}(1+v)^{2g}
   \frac{(uv)^{g-1}}{(1-uv)^2}  \Bigg] \\ 
  =& (1+u)^{2g}(1+v)^{2g}
  \coeff_{x^0}\frac{(1+ux)^{g}(1+vx)^{g}}{(1-x)(1-uvx)x^{d_1-d_2-n}}
  \cdot
   \\ &
  \cdot \Bigg[
   \frac{(uv)^{g-1}(1+u)^{g}(1+v)^{g}}{(1-uv)^2}
    \Bigg(\frac{(uv)^{N_1+N_2}-(uv)^{2N_1}}{1-(uv)^{-1}x} -
    \frac{(uv)^{2N_2+2}x-(uv)^{N_1+N_2+2}x}{1-(uv)^2x} \Bigg) \\
 &+( (uv)^{2N_2}-(uv)^{2N_1} )
 \frac{(1+u^2v)^{g}(1+uv^2)^g -
 (uv)^{g-1}(1+u)^{g}(1+v)^g}{(1-uv)^2(1-(uv)^2)} \Bigg] \\ 
  =&
(1+u)^{2g}(1+v)^{2g}
  \coeff_{x^0}\frac{(1+ux)^{g}(1+vx)^{g}}{(1-x)(1-uvx)x^{d_1-d_2-n}}
  \cdot
   \\ &
  \cdot \Bigg[
   \frac{(uv)^{g-1}(1+u)^{g}(1+v)^{g}}{(1-uv)^2}
    \Bigg(-\frac{(uv)^{2N_1}}{1-(uv)^{-1}x} -
    \frac{(uv)^{2N_2+2}x}{1-(uv)^2x}
  - \frac{(uv)^{2N_2}-(uv)^{2N_1} }{1-(uv)^2} (1-uv) \Bigg) \\
 & +
  \frac{(uv)^{g-1}(1+u)^{g}(1+v)^{g}}{(1-uv)^2}
    \Bigg(\frac{(uv)^{N_1+N_2}}{1-(uv)^{-1}x} +
    \frac{(uv)^{N_1+N_2+2}x}{1-(uv)^2x}  \Bigg) \\
    &+( (uv)^{2N_2}-(uv)^{2N_1} )
  \frac{(1+u^2v)^{g}(1+uv^2)^g-(uv)^{g}(1+u)^{g}(1+v)^g}{(1-uv)^2(1-(uv)^2)}
 \Bigg]\,,
 \end{aligned}
 $$
 hence
  $$
     \begin{aligned}
  C_n
  =&
(1+u)^{2g}(1+v)^{2g}
  \coeff_{x^0}\frac{(1+ux)^{g}(1+vx)^{g}}{(1-x)(1-uvx)x^{d_1-d_2-n}}
  \cdot
   \\ &
  \Bigg[
   - \frac{(uv)^{g-1}(1+u)^{g}(1+v)^{g}}{(1-uv)^2(1+uv)}
    \Bigg(\frac{(uv)^{2N_1+1} (1+(uv)^{-2}x) }{1-(uv)^{-1}x} +
    \frac{(uv)^{2N_2} (1+(uv)^3x)}{1-(uv)^2x} \Bigg) \\
 & +
  \frac{(uv)^{g-1}(1+u)^{g}(1+v)^{g}}{(1-uv)^2}
   (uv)^{N_1+N_2} \frac{1 -(uv)x^{2}}{(1-(uv)^{-1}x)(1-(uv)^2x)} \\
    &+( (uv)^{2N_2}-(uv)^{2N_1} )
  \frac{(1+u^2v)^{g}(1+uv^2)^g-(uv)^{g}(1+u)^{g}(1+v)^g}{(1-uv)^2(1-(uv)^2)}
 \Bigg]\,.
\end{aligned}
$$
The result follows.
\end{proof}

Now we are ready to prove our first main result, Theorem
\ref{thm:main-1}.

\begin{theorem} \label{thm:main}
  Let $\s>\s_m$ be a non-critical value. Set
   $$
   n_0=\left[\frac{\s+d_1+d_2}{2}\right] \qquad \text{and}\qquad \hn= 2\left[\frac{n_0+1}2 \right]\,.
   $$
  Then the Hodge polynomial of $\cN_\s =\cN_\s(3,1,d_1,d_2)$ is
  $$
  \begin{aligned}
   e(\cN_\s) =&(1+u)^{2g}(1+v)^{2g}
   \coeff_{x^0}\frac{(1+ux)^{g}(1+vx)^{g}}{(1-x)(1-uvx)x^{d_1-d_2}}
   \cdot
   \\
    & \cdot \Bigg[
   \left(  \frac{(uv)^{2d_1-2d_2-2 n_0} x^{n_0}}{1-(uv)^{-2}x} -
   \frac{(uv)^{2g-2-2d_1+3n_0} x^{n_0}}{1-(uv)^{3}x} \right)
   \cdot  \frac{(1+u^2v)^{g}(1+uv^2)^g  -(uv)^{g}(1+u)^{g}(1+v)^{g}}
    {(1-uv)^2(1-(uv)^2)} \\ &  +
   \frac{(uv)^{g-1}(1+u)^{g}(1+v)^{g}}{(1-uv)^2(1+uv)}
    \Bigg(
  \frac{(uv)^{2d_1-2d_2-2\hn+1} x^\hn}{(1-(uv)^{-2}x)(1-(uv)^{-1}x)}  \\
    &
    \qquad \qquad + \frac{(uv)^{2g-2-2d_1+3\hn} x^{\hn} }{(1-(uv)^{3}x)(1-(uv)^2x)}
    - \frac{(1+uv)(uv)^{g-1-d_2+\hn/2} x^{\hn}}{(1-(uv)^2x)(1-(uv)^{-1}x)}
    \Bigg) \Bigg]\, .
  \end{aligned}
  $$
\end{theorem}

\begin{proof}
 We have the following telescopic sum
  $$
   e(\cN_\s) = \sum_{\s_c>\s} (e(\cN_{\scm}) -e(\cN_{\scp})) =
   \sum_{n\geq n_0} (- C_n)\, ,
  $$
  since $\s_c=2n-d_1-d_2> \s$ is equivalent to $n>\frac{\s+d_1+d_2}{2}$, i.e.
  $n\geq n_0$. Note, incidentally, that $\s$ is non-critical is equivalent to $\frac{\s+d_1+d_2}{2}$
  not being an integer. Using Propositions \ref{prop:Cnodd} and
  \ref{prop:Cneven}, we get
  $$
  \begin{aligned}
   e(\cN_\s) =&   \sum_{n\geq n_0} (- C_n)= \sum_{n\geq n_0, n \text{ odd}} (- C_n) +\sum_{n\geq n_0, n \text{ even}}
   (- C_n) \\
   = & (1+u)^{2g}(1+v)^{2g}
   \coeff_{x^0}\frac{(1+ux)^{g}(1+vx)^{g}}{(1-x)(1-uvx)x^{d_1-d_2}}
   \cdot
   \\
    & \Bigg[
   \sum_{n\geq n_0} x^n
   \left( (uv)^{2N_1}-(uv)^{2N_2} \right)
   \cdot  \frac{(1+u^2v)^{g}(1+uv^2)^g  -(uv)^{g}(1+u)^{g}(1+v)^{g}}
    {(1-uv)^2(1-(uv)^2)} + \\ & \,  +
   \sum_{n\geq n_0, n \text{ even} }  x^n
   \frac{(uv)^{g-1}(1+u)^{g}(1+v)^{g}}{(1-uv)^2(1+uv)}
    \Bigg(\frac{(uv)^{2N_1+1} (1+(uv)^{-2}x) }{1-(uv)^{-1}x} +
    \frac{(uv)^{2N_2} (1+(uv)^3x)}{1-(uv)^2x} \\
 & \qquad \qquad\qquad \qquad - \frac{(uv)^{N_1+N_2} (1+uv)(1 -(uv)x^{2})}{(1-(uv)^{-1}x)(1-(uv)^2x)}\Bigg)
 \Bigg]\,.
  \end{aligned}
   $$
    Recall that $N_1=d_1-d_2-n$, $N_2=g-1+3n/2-d_1$. Also, note that $\hn$ is the
    first even number greater than or equal to $n_0$. We
    substitute
  $$
  \begin{aligned}
     \sum_{n\geq n_0}x^n ((uv)^{2N_1}-(uv)^{2N_2}) &
     = \sum \Big( x^n (uv)^{2d_1-2d_2-2n}- x^n(uv)^{2g-2-2d_1+3n}\Big)  \\ & =
 \frac{(uv)^{2d_1-2d_2-2 n_0} x^{n_0}}{1-(uv)^{-2}x} -\frac{(uv)^{2g-2-2d_1+3 n_0}
 x^{n_0}}{1-(uv)^{3}x}\, , \\
 \sum_{n\geq n_0, n \text{ even} } x^n (uv)^{2N_1+1} &= \sum x^n (uv)^{2d_1-2d_2-2n+1} =
     \frac{(uv)^{2d_1-2d_2-2\hn+1} x^{\hn}}{1-(uv)^{-4}x^2} \, ,\\
    \sum_{n\geq n_0, n \text{ even} } x^n (uv)^{2N_2} &= \sum x^n (uv)^{2g-2-2d_1+3n} =
     \frac{(uv)^{2g-2-2d_1+3\hn} x^{\hn}}{1-(uv)^{6}x^2} \, ,\\
   \sum_{n\geq n_0, n \text{ even} } x^n (uv)^{N_1+N_2} &=\sum x^n (uv)^{g-1-d_2+n/2} =
     \frac{(uv)^{g-1-d_2+\hn/2} x^{\hn}}{1- uv x^2}\, ,
 \end{aligned}
   $$
   in the formula above, to get
  $$
  \begin{aligned}
   e(\cN_\s) =&(1+u)^{2g}(1+v)^{2g}
   \coeff_{x^0}\frac{(1+ux)^{g}(1+vx)^{g}}{(1-x)(1-uvx)x^{d_1-d_2}}
   \cdot
   \\
    & \Bigg[
   \left( \frac{(uv)^{2d_1-2d_2-2 n_0} x^{n_0}}{1-(uv)^{-2}x}  -
   \frac{(uv)^{2g-2-2d_1+3n_0} x^{n_0}}{1-(uv)^{3}x}
   \right)
   \cdot  \frac{(1+u^2v)^{g}(1+uv^2)^g  -(uv)^{g}(1+u)^{g}(1+v)^{g}}
    {(1-uv)^2(1-(uv)^2)} \\ &  +
   \frac{(uv)^{g-1}(1+u)^{g}(1+v)^{g}}{(1-uv)^2(1+uv)}
    \Bigg(
  \frac{(uv)^{2d_1-2d_2-2\hn+1} x^\hn (1+(uv)^{-2}x)}{(1-(uv)^{-4}x^2)(1-(uv)^{-1}x)}  \\
    &
    \qquad + \frac{(uv)^{2g-2-2d_1+3\hn} x^{\hn} (1+(uv)^3x)}{(1-(uv)^{6}x^2)(1-(uv)^2x)}
    - \frac{(1+uv)(uv)^{g-1-d_2+\hn/2} x^{\hn}}{(1-(uv)^2x)(1-(uv)^{-1}x)}
    \Bigg) \Bigg] \, .
  \end{aligned}
   $$
Simplifying we get the required result. 
\end{proof}

%
%

\section{Hodge polynomial of the moduli space of rank $3$
stable bundles} \label{sec:rank3stable}

Now we want to use Proposition \ref{prop:moduli-small} to compute
the Hodge polynomial for the moduli space $M(3,1)$. Note that
$M(3,1)\cong M(3,-1)$, via $E\mapsto E^*$. Also $M(3,d)\cong
M(3,d+3k)$, for any $k\in\ZZ$, by twisting with a fixed line
bundle of degree $k$. Therefore all $M(3,d)$, with $d\not\equiv
0\pmod 3$ are isomorphic to each other.

Here we apply our previously gathered knowledge about
$\cN_\s(3,1,d_1,d_2)$ to find the Hodge polynomial of $M(3,d)$,
thus proving Theorem \ref{thm:main-2}.

\begin{theorem}\label{thm:main2}
Assume that $d\not\equiv 0 \pmod 3$. Then the Hodge polynomial of
$M(3,d)$ is
  $$
  \begin{aligned}
   e(M(3,d))
    =& \frac{(1+u)^{2g}(1+v)^{2g}}{(1-uv)(1-(uv)^2)^2(1-(uv)^3)}
   \Big( {(1+u)^{g}(1+v)^{g}(1+uv)^2(uv)^{2g-1}
   (1+u^2v)^g(1+uv^2)^g}\\
   & -(1+u)^{2g}(1+v)^{2g} (uv)^{3g-1}(1+uv+u^2v^2)
   +(1+u^2v^3)^g(1+u^3v^2)^g(1+u^2v)^g(1+uv^2)^g
   \Big)\, .
  \end{aligned}
   $$
\end{theorem}

\begin{proof}
We choose $d_2=0$, $d_1=6g-5$. By Proposition
\ref{prop:moduli-small},
 $$
 \cN_{\s_m^+}=\cN_{\s_m^+}(3,1,d_1,d_2) \to \Jac^{d_2} X\times M(3,d_1)
 $$
is a projective bundle with projective fibers of dimension
$d_1-3(g-1)-1=3g-3$. Then
 $$
  e(\cN_{\s_m^+} ) = e(M(3,d_1))\,  (1+u)^g(1+v)^g
  \frac{1-(uv)^{3g-2}}{1-uv}\, .
 $$

To compute $e(\cN_{\s_m^+})$, apply Theorem \ref{thm:main}, with
$\s=\s_m^+=\frac{d_1}3 +\epsilon$ ($\epsilon>0$ small), so
$n_0=[\frac{\sigma+d_1+d_2}2] = [ \frac23 d_1]+1=4g-3$,
$\hn=4g-2$, and $d_1-d_2-n_0=2g-2$, to get
   $$
  \begin{aligned}
   e(\cN_{\s_m^+}) =&(1+u)^{2g}(1+v)^{2g}
   \coeff_{x^0}\frac{(1+ux)^{g}(1+vx)^{g}}{(1-x)(1-uvx)x^{2g-2}}
   \cdot
   \\
    & \qquad \cdot \Bigg[
   \left(
   \frac{(uv)^{4g-4}}{1-(uv)^{-2}x} - \frac{(uv)^{2g-1}}{1-(uv)^{3}x}  \right)
   \cdot  \frac{(1+u^2v)^{g}(1+uv^2)^g  -(uv)^{g}(1+u)^{g}(1+v)^{g}}
    {(1-uv)^2(1-(uv)^2)} \Bigg] \\ &
     + (1+u)^{2g}(1+v)^{2g}
   \coeff_{x^0}\frac{(1+ux)^{g}(1+vx)^{g}}{(1-x)(1-uvx)x^{2g-1}}
   \cdot
   \\
  & \qquad \cdot \Bigg[\frac{(uv)^{g-1}(1+u)^{g}(1+v)^{g}}{(1-uv)^2(1+uv)}
    \Bigg(
  \frac{(uv)^{4g-5} }{(1-(uv)^{-2}x)(1-(uv)^{-1}x)}  \\
    &
    \qquad \qquad + \frac{(uv)^{2g+2}  }{(1-(uv)^{3}x)(1-(uv)^2x)}
    - \frac{(1+uv)(uv)^{3g-2} }{(1-(uv)^2x)(1-(uv)^{-1}x)}
    \Bigg) \Bigg]\, .
  \end{aligned}
   $$
Introducing the notation
   $$
  \begin{aligned}
   F_1(a,b,c) &= \Res_{x=0} \frac{(1+ux)^{g}(1+vx)^g}{(1-ax)(1-bx)(1-cx)x^{2g-1}}
  \, ,\\
  F_2(a,b,c,d) &= \Res_{x=0}
  \frac{(1+ux)^{g}(1+vx)^g}{(1-ax)(1-bx)(1-cx)(1-dx)x^{2g-2}}\, ,
  \end{aligned}
   $$
we write
   \begin{equation}\label{eqn:F3}
  \begin{aligned}
   e(\cN_{\s_m^+}) =& (1+u)^{2g}(1+v)^{2g}
   \Bigg[
   \Big(   {(uv)^{4g-4}} F_1(1,uv,{(uv)^{-2}}) - (uv)^{2g-1} F_1( 1,uv ,{(uv)^{3}}) \Big) \cdot \\ &
   \qquad  \cdot  \frac{(1+u^2v)^{g}(1+uv^2)^g  -(uv)^{g}(1+u)^{g}(1+v)^{g}}
    {(1-uv)^2(1-(uv)^2)} \\ &  +
   \frac{(uv)^{g-1}(1+u)^{g}(1+v)^{g}}{(1-uv)^2(1+uv)}
    \Big(
  {(uv)^{4g-5} } F_2(1,uv, (uv)^{-2},(uv)^{-1})  \\
    &
    \qquad + {(uv)^{2g+2}  }{F_2(1,uv, (uv)^{3}, (uv)^2)}
    - (1+uv)(uv)^{3g-2} {F_2(1,uv, (uv)^2,(uv)^{-1})}
    \Big) \Bigg] \,.
  \end{aligned}
  \end{equation}

In the proof of \cite[Proposition 8.1]{MOV1}, we computed
 \begin{equation}\label{eqn:F1abc}
 F_1(a,b,c)=
    \frac{(a+u)^{g}(a+v)^{g}}{(a-b)(a-c)}+
    \frac{(b+u)^{g}(b+v)^{g}}{(b-a)(b-c)}+
    \frac{(c+u)^{g}(c+v)^{g}}{(c-a)(c-b)}\, .
 \end{equation}
Also, the function
 $$
  G(x)=\frac{(1+ux)^{g}(1+vx)^g}{(1-ax)(1-bx)(1-cx)(1-dx)x^{2g-2}}\,
  ,
  $$
is meromorphic on $\CC\cup \{\infty\}$ with poles at $x=0$,
$x=1/a$, $x=1/b$, $x=1/c$, $x=1/d$, and no pole at $\infty$. So
  $$
  F_2(a,b,c,d)=-\Res_{x=1/a} G(x) -\Res_{x=1/b} G(x) -\Res_{x=1/c}
  G(x)-\Res_{x=1/d} G(x)\, ,
  $$
from where
   \begin{equation}\label{eqn:F2abcd}
    F_2(a,b,c,d)=  \frac{(a+u)^{g}(a+v)^{g}}{(a-b)(a-c)(a-d)}+
    \frac{(b+u)^{g}(b+v)^{g}}{(b-a)(b-c)(b-d)}+
    \frac{(c+u)^{g}(c+v)^{g}}{(c-a)(c-b)(c-d)}+
    \frac{(d+u)^{g}(d+v)^{g}}{(d-a)(d-b)(d-d)}\, .
  \end{equation}

Using (\ref{eqn:F1abc}) and (\ref{eqn:F2abcd}) into
(\ref{eqn:F3}), we get
  $$
  \begin{aligned}
   e(M(3,d_1)) =& \frac{e(\cN_{\s_m^+}) (1-uv)}{(1+u)^g(1+v)^g (1-(uv)^{3g-2})} \\
    =& \frac{ (1+u)^{g}(1+v)^{g}}{(1-uv)(1-(uv)^2)^2(1-(uv)^3)}
   \Big( -(1+u)^{g}(1+v)^{g}(1+uv)^2(uv)^{2g-1}
   (1+u^2v)^g(1+uv^2)^g \\
   & +(1+u)^{2g}(1+v)^{2g} (uv)^{3g-1}(1+uv+u^2v^2)
   +(1+u^2v^3)^g(1+u^3v^2)^g(1+u^2v)^g(1+uv^2)^g
   \Big)\, .
  \end{aligned}
   $$
\end{proof}


\begin{thebibliography}{MMMM}

%


%

%

%
\bibitem{BGP}\textsc{Bradlow, S. B.; Garc\'{\i}a--Prada, O.}:
{Stable triples, equivariant bundles and dimensional reduction}.
\textsl{Math. Ann.} \textbf{304} (1996) 225--252.

\bibitem{BGPG}\textsc{Bradlow, S. B.; Garc\'{\i}a--Prada, O.;
Gothen, P.B}: {Moduli spaces of holomorphic triples over compact
Riemann surfaces}. \textsl{Math. Ann.} \textbf{328} (2004)
299--351.

%

\bibitem{Bur}\textsc{Burillo, J.}: {El polinomio de Poincar\'{e}--Hodge de un
producto sim\'{e}trico de variedades k\"{a}hlerianas compactas}.
\textsl{Collect. Math.} \textbf{41} (1990) 59--69.

\bibitem{BaR}\textsc{Del Ba\~{n}o, S.}:
On the motive of moduli spaces of rank two vector bundles over a
curve. \textsl{Compositio Math.} \textbf{131} (2002) 1--30.

\bibitem{De} \textsc{Deligne, P.:} {Th\'{e}orie de Hodge I,II,III}.
In \textsl{Proc.\ I.C.M.}, vol.\ 1, 1970, pp. 425--430; in
\textsl{Publ.\ Math.\ I.H.E.S.} \textbf{40} (1971) 5--58; ibid.\
\textbf{44} (1974) 5--77.



\bibitem{Du} \textsc{Durfee, A.H.:} \textit{Algebraic
varieties which are a disjoint union of subvarieties},
\textsl{Lecture Notes in Pure Appl.\ Math.} \textbf{105}, Marcel
Dekker, 1987, pp. 99--102.

\bibitem{DK} \textsc{Danivol, V.I.; Khovanski\v{\i},
A.G.:} {Newton polyhedra and an algorithm for computing
Hodge-Deligne numbers}, \textsl{Math.\ U.S.S.R. Izvestiya}
\textbf{29} (1987) 279--298.

%
\bibitem{EK} \textsc{Earl, R.; Kirwan, F.:} {The Hodge numbers of the
moduli spaces of vector bundles over a Riemann surface}.
\textsl{Q. J. Math.} \textbf{51} (2000) 465--483.

%

\bibitem{GP}\textsc{Garc\'{\i}a--Prada, O.}: {Dimensional reduction of
stable bundles, vortices and stable pairs}. \textsl{Internat. J.
Math.} \textbf{5} (1994) 1--52.

\bibitem{GPGM}\textsc{Garc\'{\i}a--Prada, O.; Gothen, P.B.; Mu\~{n}oz, V.}:
{Betti numbers of the moduli space of rank $3$ parabolic Higgs
bundles}. \textsl{Memoirs Amer.\ Math.\ Soc.} {\bf 187} (2007).

%
%
%
%
%
%
%
%
%


\bibitem{MOV1} \textsc{Mu\~{n}oz, V.; Ortega, D.; V\'{a}zquez-Gallo, M-J.}:
Hodge polynomials of the moduli spaces of pairs.
 \textsl{Internat. J. Math.} \textbf{18} (2007) 695--721.

\bibitem{MOV2} \textsc{Mu\~{n}oz, V.; Ortega, D.; V\'{a}zquez-Gallo, M-J.}:
Hodge polynomials of the moduli spaces of triples of rank~$(2,2)$.
Preprint {\tt arXiv:math/0701642}.


%
%
%
\bibitem{Sch} \textsc{Schmitt, A.}:
{A universal construction for the moduli spaces of decorated
vector bundles}. \textsl{Transform. Groups} \textbf{9} (2004)
167--209.

\bibitem{Th}\textsc{Thaddeus, M.}:
{Stable pairs, linear systems and the Verlinde formula}.
\textsl{Invent. Math.} \textbf{117} (1994) 317--353.

\bibitem{Zagier} \textsc{Zagier, D.}: {Elementary aspects of the Verlinde formula and of the
Harder-Narasimhan-Atiyah-Bott formula}. \textsl{Proceedings of the
Hirzebruch 65 Conference on Algebraic Geometry (Ramat Gan, 1993),
Israel Math. Conf. Proc.} \textbf{9} (1996) 445--462.

\end{thebibliography}
\end{document}